\documentclass[a4paper,11pt]{amsart}
\usepackage[basic]{mnotes}
\usepackage{graphicx,xcolor}
\usepackage{amsmath,amssymb,amsthm}
\usepackage{enumitem}
\usepackage{mathtools}
\usepackage{tikz}
\usepackage{stmaryrd}
\usepackage{bbold}
\usepackage{dsfont}
\usepackage[OT1]{fontenc}
\usepackage[normalem]{ulem}

\newtheorem{thm}{Theorem}[section]
\newtheorem*{thm*}{Theorem}
\newtheorem{cor}[thm]{Corollary}
\newtheorem{lem}[thm]{Lemma}
\newtheorem{claim}[thm]{Claim}
\newtheorem{prop}[thm]{Proposition}
\newcounter{introthm}

\newtheorem{introtheo}[introthm]{Theorem}
\theoremstyle{definition}

\newtheorem{fact}[]{Fact}
\theoremstyle{remark}
\newtheorem{rem}[thm]{Remark}
\newtheorem{exmp}[thm]{Example}

\newcommand{\R}{\mathds{R}}

\newcommand{\N}{\mathds{N}}

\newcommand{\I}{\mathds{I}}
\newcommand{\J}{\mathds{J}}

\newcommand{\mc}{\mathcal}
\newcommand{\eps}{\varepsilon}
\newcommand{\ssq}{\subseteq}
\renewcommand{\:}{\colon}

\DeclareMathOperator{\Per}{per}
\DeclareMathOperator{\PER}{PER}

\DeclareMathOperator{\Fix}{Fix}

\renewcommand{\d}{\partial}
\newcommand{\Id}{\operatorname{Id}}
\newcommand{\0}{\mathbb 0}
\newcommand{\1}{\mathbb 1}

\title{Period spectrum and rigidity of interval maps}
\author[Gabriel Fuhrmann]{Gabriel Fuhrmann}
\author[Maik Gr\"{o}ger]{Maik Gr\"{o}ger}
\author[Alejandro Passeggi]{Alejandro Passeggi}

\address[Gabriel Fuhrmann]{Department of Mathematical Sciences, Durham University, UK}
\email{gabriel.fuhrmann@durham.ac.uk}

\address[Maik Gr\"{o}ger]{Faculty of Mathematics and Computer Science, Jagiellonian University in Krak\'ow, Poland}
\email{maik.groeger@im.uj.edu.pl}

\address[Alejandro Passeggi]{CMAT, Facultad de Ciencias, Universidad de la Rep\'{u}blica,
Igu\'{a} 4225, 11400 Montevideo, Uruguay} \email{apasseggi@cmat.edu.uy}

\begin{document}

\begin{abstract}
We study the extent to which periodic orbits determine a one-dimensional dynamical system.
In particular, we show that they determine every topologically mixing interval map.
More generally, for a non-mixing interval map $f$ with dense periodic points, the periodic orbits need not determine $f$, but they always do so for $f^2$.
Finally, for maps on compact intervals, we obtain a local version of this rigidity: for each basic set $B$ arising in the spectral decomposition of an interval map $f$, there exists an iterate $f^q$ whose restriction to $B$ is determined by the periodic orbits of $f$.
\end{abstract}

\maketitle

\section{Introduction}

Rigidity phenomena are a recurring and fundamental theme in mathematics.
A classic example is Kac's question, ``Can one hear the shape of a drum?'' \cite{Kac1966}, which asks whether the geometry of a domain is determined by the spectrum of its Laplacian.
For further details, see \cite{Zelditch2004} and the references therein.

A well-established strategy for different instances of rigidity problems is to analyse the periodic orbits of a naturally associated dynamical system.
In the context of Kac's question, the relevant dynamics for planar domains are given by billiard flows; see, e.g., \cite{KaloshinSorrentino2022} for a recent overview.
For closed Riemannian manifolds, one considers the geodesic flow, with major developments especially in the negatively curved setting (see, for instance, \cite{GuillarmouLefeuvre2019}).
In either case, it is largely the chaotic nature of the associated dynamics that drives the observed rigidity.

This leads naturally to a rigidity problem in dynamics in its own right:
\begin{center}
    To what extent do the periodic orbits of a chaotic map determine the map itself?
\end{center}
Let us emphasise that this question concerns the collection of periodic orbits as subsets of the phase space, without any additional combinatorial information; in chaotic regimes, the question would otherwise become trivial.

The above problem arises naturally in several classical settings of chaotic dynamics, including Anosov and Axiom~A systems, pseudo-Anosov maps, and related variants.
To approach it in any degree of generality, interval dynamics is the natural place to begin: many of the fundamental mechanisms of chaotic behaviour are already present in dimension one, and it is precisely in this setting that we obtain general positive answers.

The following statement illustrates the flavour of our results.
Here and throughout, an \emph{interval map} is a continuous map $f\colon \I\to \I$, where $\I$ is a non-degenerate interval; $\PER(f)$ denotes the \emph{period spectrum}, that is, the set of periodic orbits of $f$.

\begin{thm*}
Let $f$ and $g$ be topologically mixing interval maps.
If $\PER(f)=\PER(g)$, then $f=g$.
\end{thm*}

Recall that a map $f$ is \emph{topologically mixing} if, for any non-empty open sets $U,V\subseteq \I$, there is $N\in\N$ such that $f^n(U)\cap V\neq\emptyset$ for all $n\geq N$.
For interval maps, chaotic behaviour is closely tied to topological mixing, as shown by Blokh's spectral decomposition theorem \cite{Blokh1995}.
In this sense, topological mixing lies at the heart of chaos on the interval and, accordingly, plays a central role in this article.

By contrast with the above theorem, it is not difficult to construct two non-mixing maps with dense periodic points that share the same periodic orbits (see Figure~\ref{fig: transitive and conjugate maps with same spectrum with reverse tent map}).

\begin{figure}[h]
    \centering
    \begin{tikzpicture}[scale=3.8, declare function={g(\x)=0.5 * abs(4 * \x - 1); T1(\x)=ifthenelse(\x<=0.5, 1 - g(\x), 1 - \x); T2(\x)=ifthenelse(\x<=0.5, 1 - \x, g(1 - \x));}]
    
    \begin{scope}[local bounding box=left]
        \draw[-] (0,0) -- (1.1,0);
        \draw[-] (0,0) -- (0,1.1);
        \draw[lightgray] (0,0) -- (1,1);
        \draw[dashed, lightgray] (0,0.5) -- (1,0.5);
        \draw[dashed, lightgray] (0,1) -- (1,1);
        \draw[dashed, lightgray] (1,0) -- (1,1);
        \draw[dashed, lightgray] (0.5,0) -- (0.5,1);
        \draw[thick] plot[domain=0:1, samples=101] (\x, {T1(\x)});
        \foreach \X in {0,1}
        {\draw (\X,0.05) -- (\X,-0.05) node[below]{$\X$};}
        \draw (0.05,1) -- (-0.05,1) node[left]{$1$};
        \node[above left] at (0.5,0.8) {$F_1$};
    \end{scope}
    
    \begin{scope}[xshift=2.0cm, local bounding box=right]
        \draw[-] (0,0) -- (1.1,0);
        \draw[-] (0,0) -- (0,1.1);
        \draw[lightgray] (0,0) -- (1,1);
        \draw[dashed, lightgray] (0,0.5) -- (1,0.5);
        \draw[dashed, lightgray] (0,1) -- (1,1);
        \draw[dashed, lightgray] (1,0) -- (1,1);
        \draw[dashed, lightgray] (0.5,0) -- (0.5,1);
        \draw[thick] plot[domain=0:1, samples=101] (\x, {T2(\x)});
        \foreach \X in {0,1}
        {\draw (\X,0.05) -- (\X,-0.05) node[below]{$\X$};}
        \draw (0.05,1) -- (-0.05,1) node[left]{$1$};
        \node[above left] at (0.38,0.8) {$F_2$};
    \end{scope}
    \end{tikzpicture}
    \caption{
        An example of two interval maps $F_1$ and $F_2$ that are not topologically mixing but have dense periodic points and the same period spectrum (see Example~\ref{exmp: transitive and conjugate maps with same spectrum}). 
        While $F_1$ and $F_2$ are conjugate, the same phenomenon can also occur for non-conjugate maps on $[0,1]$; see Example~\ref{exmp: non-conjugate maps with same spectrum}.
    }
    \label{fig: transitive and conjugate maps with same spectrum with reverse tent map}
\end{figure}
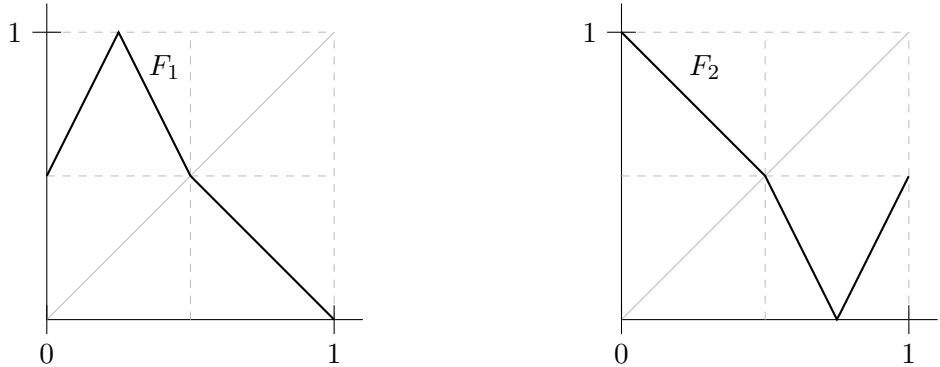

We now state our main results.
Throughout, we follow the convention that periods are always minimal.
That is, we call $x$ \emph{$p$-periodic} if $f^p(x)=x$ and $f^q(x)\neq x$ for $1\leq q<p$; accordingly, we call $x$ \emph{prime periodic} if it is $p$-periodic for some prime $p$. 
We write $\Per_p(f)\subseteq \I$ for the set of points with period $p$.

The preceding theorem was stated separately above for emphasis; our first main result now subsumes it.

\begin{introtheo}\label{introthm: mixing maps are determined by period spectrum}
    Suppose $f, g \colon \I\to\I$ are continuous maps with $\PER(f)=\PER(g)$.
    If the set $\bigcup_{p \text{ prime}} \Per_p(f)$ of prime periodic points is dense in $\I$, then $f=g$.

    In particular, if $f$ is topologically mixing and $\PER(f)=\PER(g)$, then $f=g$.
\end{introtheo}

We remark that, under additional assumptions, the proof of Theorem~\ref{introthm: mixing maps are determined by period spectrum} essentially yields an algorithm for approximating $f$ on compact subintervals from its period spectrum $\PER(f)$.
Specifically, every periodic point of $f$ must correspond to a proper crossing, that is, if $x$ is $p$-periodic, then the graph of $f^p$ crosses the diagonal at $x$, and, for every $p\in\N$, the set $\Per_p(f)$ must intersect each compact subinterval in only finitely many points.
The latter condition is satisfied, for instance, by any continuous piecewise monotone map on a compact interval.

Although distinct interval maps may share the same period spectrum even when their periodic points are dense, we obtain the following rigidity result for the second iterate.

\begin{introtheo}\label{introthm: dense periodic orbits determine second iterate}
    Suppose $f, g \colon \I\to\I$ are continuous maps with $\PER(f)=\PER(g)$.
    If the set $\bigcup_{p \in \N} \Per_p(f)$ of all periodic points is dense in $\I$, then $f^2=g^2$.
\end{introtheo}

We obtain the above statement as a consequence of Theorem~\ref{introthm: mixing maps are determined by period spectrum}, together with the classical fact that the second iterate of any interval map with dense periodic points admits a decomposition into intervals on each of which it is either topologically mixing or the identity---in the compact case, this goes back originally to Barge and Martin \cite{BargeMartin1985}.

Even beyond dense periodicity, periodic orbits are known to carry substantial information about the dynamics of interval maps. 
A classical manifestation of this is Sharkovsky's theorem on the coexistence of periodic orbits \cite{Sharkovsky1964b}, together with its refinements describing how the combinatorics of one periodic orbit can force combinatorial features of others; see \cite{MisiurewiczNitecki1991} and the references therein.

One reason for the central role of periodic orbits is that, for interval maps, the closure of the union of all periodic orbits coincides with the closure of the recurrent points and hence contains the support of every invariant measure of the map---another foundational result due to Sharkovsky \cite{Sharkovsky1964}, later independently rediscovered by Young \cite{Young1979} and by Coven and Hedlund \cite{CovenHedlund1980}.

In a much subtler and more refined sense, this is also reflected in Blokh's spectral decomposition theorem \cite{Blokh1995}.
One of its consequences is that the positive entropy of a continuous map $f$ on a compact interval is always localised on so-called \emph{basic sets}, that is, maximal $\omega$-limit sets $B$ such that the restriction of an iterate of $f$ to $B$ factors onto a mixing interval map. 
This is reminiscent of, and closely related to, the fact that entropy is always realised by one-dimensional horseshoes \cite{Misiurewicz1979,Misiurewicz1980,Buzzi2006}.

Our final main result concerns these basic sets.
\begin{introtheo}\label{introthm: basic sets}
    Suppose $f:\I\to \I$ and $g:\I\to \I$ are continuous maps with $\I$ compact and $\PER(f)=\PER(g)$. 
    Then each basic set $B$ of $f$ is a basic set of $g$ (and vice versa).
    Further, there exists $q\in \N$ such that
    $\left.f^q\right|_B=\left.g^q\right|_B$.
\end{introtheo}

As we will see, the preceding theorem is itself a consequence of the more general local result, Theorem~\ref{thm: main theorem on basic sets}: if two maps $f$ and $g$ on a compact interval share a basic set $B$ and have the same periodic orbits on $B$, then there exists $q\in\N$ such that $f^q$ and $g^q$ coincide on $B$.

We further note that Theorem~\ref{introthm: basic sets} implies that two maps with the same periodic orbits must have the same topological entropy. 
In contrast with the piecewise monotone setting \cite{Parry1966,MisiurewiczSzlenk1977,MisiurewiczSzlenk1980,MilnorThurston1988}, we obtain this conclusion without appealing to a priori growth formulae for periodic points, which are well known to break down in general, for instance when there are infinitely many periodic orbits of a given period.

\subsection{Outline}
This article is organised as follows. 
In Section~\ref{sec: left right behaviour determines permutation},
we prove an entirely combinatorial result on permutations $\sigma$
of totally ordered sets $S$ of prime cardinality $p$.
This result states that the \emph{left-right behaviour} of $\sigma$, that is, whether each point
$x\in S$ is mapped to its right or to its left under $\sigma^n$ for 
$n=1,\ldots,p-1$ uniquely determines $\sigma$, provided $\sigma$ is 
a $p$-cycle.  

In our analysis of maps $f$ on $\I$, the role of $S$ is played by 
periodic orbits of prime period, and $\sigma$ is given by 
$\left.f\right|_S$. 
It is then clear that if the union of the prime periodic orbits is dense and the left-right behaviour of $f$ on each of them is known, then $f$ is determined.
In other words, the question of whether the period spectrum determines 
$f$ reduces to whether it determines the left-right behaviour. 
The latter question is settled in 
Section~\ref{sec: periodic points on the interval determine left-right behaviour}, 
where we make use of a technical lemma whose proof we defer to 
Section~\ref{sec: proof of technical lemma}.

In Section~\ref{sec: applications}, we discuss how 
Theorems~\ref{introthm: mixing maps are determined by period spectrum}--\ref{introthm: basic sets}
follow from the more technical results of the previous sections.
Towards proving Theorem~\ref{introthm: dense periodic orbits determine second iterate}, we will need to extend a decomposition result for maps on compact intervals with dense periodic points to arbitrary intervals; see Theorem~\ref{thm: CovenMulvey}.
The details of this extension are given in the appendix.
To contrast these positive results, 
Section~\ref{sec: examples} shows that the period spectrum 
does not 
determine the map in general, as illustrated in 
Figures~\ref{fig: transitive and conjugate maps with same spectrum with reverse tent map}
and \ref{fig: non-transitive and non-conjugate maps with same spectrum}.

\subsection*{\textbf{Declaration of generative AI use.}}
All ideas, conceptual contributions, and proofs are the authors' own and were developed without the use of AI.
The authors used LLMs solely to improve the phrasing and clarity of parts of the text and to assist in creating the figures.

\section{Recovering a permutation from its left-right behaviour}
\label{sec: left right behaviour determines permutation}
Given a self-map $f\colon X\to X$ on a linearly ordered set $X$ and a point $x\in X$, we define the \emph{left-right behaviour} of $x$ to be the sequence $(\ell_n(x))\in\{<,>,=\}^{\N}$ satisfying $f^n(x)\,\ell_n(x)\,x$ for all $n\in\N$. 
Our goal lies in understanding how the left-right behaviour of the periodic points of interval maps determines the combinatorics on the periodic orbits and, ultimately, the maps globally.

As a first step towards this goal, we restrict our attention in this section to single $p$-periodic orbits.
More precisely, we take $f$ to be a \emph{$p$-cycle} on a finite linearly ordered set $x_1<x_2<\cdots<x_p$, meaning that for each 
$i\in\{1,\ldots,p\}$ there exists $0\leq n\leq p-1$ such that 
$f^n(x_1)=x_i$. 
Observe that in this setting, the left-right behaviour of $f$---that is, the left-right behaviour of each point $x_i$---reduces to knowing, for each $n=1,\ldots,p-1$ and $i=2,\ldots,p-1$, whether $f^n(x_i)>x_i$ or $f^n(x_i)<x_i$.

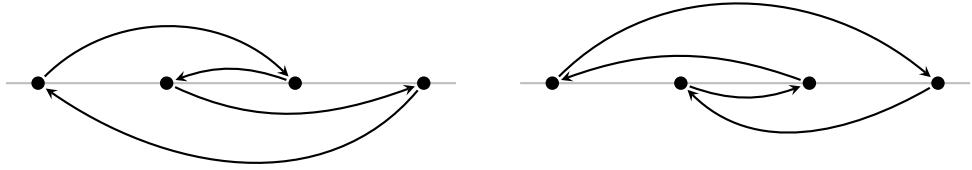
\begin{figure}[h]
    \centering
    \begin{tikzpicture}[
        scale=0.85, point/.style={circle, draw, fill=black, inner sep=1.5pt, minimum size=4pt},
        >=stealth, auto, thick
    ]
        \begin{scope}

        \draw[lightgray] (-0.5,0) -- (6.5,0);

        \node[point] (1a) at (0,0) {};
        \node[point] (2a) at (2,0) {};
        \node[point] (3a) at (4,0) {};
        \node[point] (4a) at (6,0) {};
    
        \draw[->, shorten >=1pt, shorten <=1pt] (1a) to[out=45,in=135] (3a);
        \draw[->, shorten >=1pt, shorten <=1pt] (3a) to[out=160,in=20] (2a);
        \draw[->, shorten >=1pt, shorten <=1pt] (2a) to[out=-25,in=-160] (4a);
        \draw[->, shorten >=1pt, shorten <=1pt] (4a) to[out=-130,in=-35] (1a);

        \end{scope}

        \begin{scope}[xshift=8cm]

        \draw[lightgray] (-0.5,0) -- (6.5,0);

        \node[point] (1b) at (0,0) {};
        \node[point] (2b) at (2,0) {};
        \node[point] (3b) at (4,0) {};
        \node[point] (4b) at (6,0) {};

        \draw[->, shorten >=1pt, shorten <=1pt] (1b) to[out=45,in=140] (4b);
        \draw[->, shorten >=1pt, shorten <=1pt] (4b) to[out=-150,in=-45] (2b);
        \draw[->, shorten >=1pt, shorten <=1pt] (2b) to[out=-20,in=-160] (3b);
        \draw[->, shorten >=1pt, shorten <=1pt] (3b) to[out=160,in=20] (1b);
        \end{scope}
    \end{tikzpicture}
    \caption{Two distinct permutations on four points (indicated as dots), both exhibiting the same left-right behavior (with arrows depicting the combinatorics).}
    \label{fig:example_dist_perm_four_points}
\end{figure}

As Figure~\ref{fig:example_dist_perm_four_points} illustrates, two distinct permutations on four points may share the same left-right behaviour. 
In fact, it is not difficult to see that this extends to all composite $p$: whenever $p$ is not prime, there exist 
distinct permutations on $p$ points $x_1 < \dots < x_p$ with the same left-right behaviour.
However, we have

\begin{lem}\label{lem: left-right determines map}
 Suppose $f$ is a $p$-cycle on a linearly ordered collection of finitely many points $x_1<\ldots< x_p$ with $p$ prime.
 If $g$ is a $p$-cycle on the same points and with the same left-right behaviour as $f$, then $f=g$.
\end{lem}
The proof of Lemma~\ref{lem: left-right determines map} relies on 
an auxiliary statement (Lemma~\ref{lem: prime periodic orbits}), which in turn rests on a classical combinatorial lemma due to Lyndon and Sch\"utzenberger. 
We introduce the necessary notation for the latter before proceeding.

We write $\N^+$ for the collection of finite \emph{words} over $\N$, that is, $\N^+=\bigcup_{n\in\N}\N^n$.
Given two words $u,v\in \N^+$, their concatenation is denoted by $uv$; further, for $m\in \N$, we write $u^m$ for the $m$-fold concatenation of $u$ with itself, that is, $u^m=\underbrace{uu\dots u}_{m\text{-times}}$.

\begin{lem}[{\cite[Lemma~3]{LyndonSchuetzenberger1962}}]\label{lem: Lyndon Schuetzenberger}
Suppose $u,v\in \N^+$ satisfy $uv=vu$.
Then there is $w\in \N^+$ and there are $m_1,m_2\in\N$ such that
$u=w^{m_1}$ and $v=w^{m_2}$.
\end{lem}

\begin{lem}\label{lem: prime periodic orbits}
 Let $f$ be a $p$-cycle on a linearly ordered collection of  finitely many points $x_1<\ldots< x_p$ with $p$ prime.
For all $i=1,\ldots,p-1$ and each pair $k, k'\in \{1,\ldots,i\}$, we have the following implication
\[
\{\text{for all } n \text{, we have } f^n(x_k)\leq x_i \text{ if and only if } f^n(x_{k'})\leq x_i\} \Longrightarrow k=k'.
\]

\end{lem}

\begin{proof}[Proof of Lemma~\ref{lem: prime periodic orbits}]
 Consider $i\in \{1,\ldots,p-1\}$.
 For a contradiction, suppose there are $k\neq k'\in \{1,\ldots,i\}$ such that for all $n$ we have
 $f^n(x_k)\leq x_i$ if and only if $f^n(x_{k'})\leq x_i$.
 We are going to show that this implies that $p$ is composite contrary to the assumptions.

 Let $0=t_0<t_1< \ldots<t_i=p$ be the first return times of $x_k$ (and $x_{k'}$) to $\{x_1,x_2,\ldots,x_i\}$ and set $s_j=t_{j+1}-t_{j}$ for $j=0,\ldots,i-1$.
 Set $y_j=f^{t_{j}}(x_k)$ and
 $z_j=f^{t_{j}}(x_{k'})$.
 It is obvious but important to note that $s_j$ is the (mutual) first return time of $y_j$ and $z_j$ to $\{x_1,x_2,\ldots,x_i\}$.

 Let $j_0\in\{0,\dots,i-1\}$ be the smallest index with
 \begin{align*}
  y_{j_0}=z_0(=x_{k'}) \qquad \text{or} \qquad z_{j_0}=y_0(=x_k).
 \end{align*}
 As $k\neq k'$ by assumption, $j_0>0$.
Note that if $y_{j_0}=z_0$ and $z_{j_0}=y_0$---that is, if $f^{t_{j_0}}(x_k)=x_{k'}$ and
$f^{t_{j_0}}(x_{k'})=x_k$---then $p=2 t_{j_0}$ in contradiction to $p$ being prime.

We may hence assume without loss of generality that $j_0$ satisfies $y_{j_0}=z_0$ while $z_{j_0}\neq y_0$.
Observe that since $y_{j_0}=z_0$, we have that $s_{j_0}=s_0$ and $f^{s_{j_0}}(y_{j_0})=z_1$.
Consequently, $s_{j_0+1}=s_1$ and $f^{s_{j_0+1}}(y_{j_0+1})=z_2$ and, continuing in this way, we obtain
\begin{align}\label{eq: recursive relation return times}
s_{j_0}=s_0,s_{j_0+1}=s_1,
\ldots, s_{j}=s_{j \bmod j_0},
\ldots, s_{i-1}=s_{(i-1) \bmod j_0}.
\end{align}
Let $j_1=\ell j_0+r$ be the smallest index such that $z_{j_1}=y_0$, where $r\in \{0,\ldots,j_0-1\}$.
Note that if $r=0$, then $p=(\ell+1)t_{j_0}$ in contradiction to the assumption that $p$ is prime.

Suppose hence that $r\neq 0$.
Note that, on the one hand, $y_1=f^{s_0}(y_0)$ while on the other hand,
\[y_1=z_{j_1+1}=f^{s_{j_1}}(z_{j_1})=f^{s_{\ell j_0 +r}}(z_{j_1})\stackrel{\eqref{eq: recursive relation return times}}{=}f^{s_r}(z_{j_1})=f^{s_r}(y_0)\] so that similarly to \eqref{eq: recursive relation return times}, we obtain
\[s_r=s_{0},s_{r+1}=s_{1},\ldots, s_{j_0-1}=s_{j_0-r-1}, s_{0}=s_{j_0-r},s_1=s_{j_0-r+1},\ldots,s_{r-1}=s_{j_0-1}.\]
Put differently, the words $u=s_0s_1\ldots s_{r-1}$ and $v=s_rs_{r+1}\ldots s_{j_0-1}$ in $\N^+$ satisfy $uv=vu$.
From Lemma~\ref{lem: Lyndon Schuetzenberger}, we hence obtain $m,m_1,m_2$ such that $w=s_0s_1\ldots s_{m}$ satisfies $u=w^{m_1}$ and $v=w^{m_2}$.
As $p=(\ell+1)\cdot (s_0+s_1+\ldots+s_{j_0-1})+(s_0+s_1+\ldots+s_{r-1})$, this implies $(s_0+s_1+\ldots +s_{m}) \, |\, p$ in contradiction to $p$ being prime.
\end{proof}
We now turn to the proof of Lemma~\ref{lem: left-right determines map}. 
Throughout this proof, we say that a quantity associated to $f$ 
(in particular, certain return times) is \emph{uniquely determined by the left-right behaviour} to mean that any permutation $g$ with the same left-right behaviour as $f$ assigns the same value to that quantity.
\begin{proof}[Proof of Lemma~\ref{lem: left-right determines map}]
Given $i\in\{1,\ldots,p\}$
and $k\leq i$, let $0=t^{k,i}_0<t_1^{k,i}<\ldots <t_{i-1}^{k,i}<p$ be the first return times of
$x_k$ to $\{x_1,x_2,\ldots,x_{i}\}$ and let $\ell^{k,i}_j$ be such that $f^{t_{j}^{k,i}}(x_k)=x_{\ell^{k,i}_j}$ (where $j\in\{0,\ldots,i-1\}$).
Note that knowing $t_j^{k,i}$ and ${\ell^{k,i}_j}$ for all $k\leq i$ and $j\leq i-1$ amounts to knowing
$t_j^{i,i}$ and ${\ell^{i,i}_j}$ for all $j\leq i-1$.
By a finite induction over $i$, we next show that, for each $i$, the times
$t_j^{i,i}$ and the indices ${\ell^{i,i}_j}$
are uniquely determined by the left-right behaviour of $f$.
This will finish the proof.

The base case is trivial (with $t_0^{1,1}=0$ and $\ell_0^{1,1}=1$).
Suppose that for some $i$ and for all $k\leq i$ and $j\leq i-1$, 
the times $t^{k,i}_j$ and indices $\ell^{k,i}_j$ are uniquely 
determined by the left-right behaviour of $f$.
Now, note that also the times $t_j^{i+1,i+1}$ are uniquely determined by the left-right behaviour of $f$: $t_0^{i+1,i+1}<t_1^{i+1,i+1}<\dots<t_i^{i+1,i+1}$ coincide with exactly those times $t$ where $f^t(x_{i+1})\leq x_{i+1}$.
Observe further that since $f^{t_1^{i+1,i+1}}(x_{i+1})\leq x_i$,
there necessarily is some $k\leq i$ with
$t_{j-1}^{k,i}=t_{j}^{i+1,i+1}-t_{1}^{i+1,i+1}$ for all $j=1,\ldots,i$.
By Lemma~\ref{lem: prime periodic orbits},
$k$ is unique.
Using the inductive hypothesis, it follows that also ${\ell^{i+1,i+1}_j}$ (for $j\leq i$) is uniquely determined by the left-right behaviour of $f$; specifically, $\ell_0^{i+1,i+1}=i+1$ and $\ell_j^{i+1,i+1}=\ell_{j-1}^{k,i}$ for $j\geq 1$.
\end{proof}
\begin{rem}
Given Lemma~\ref{lem: left-right determines map}, one can in principle recover the permutation $f$ algorithmically by exhausting all $p$-cycles and identifying the unique one whose left-right behaviour coincides with that of $f$. 
Since there are $(p-1)!$ distinct $p$-cycles, this brute-force approach therefore comes with factorial computational complexity. 
The proof above, however, yields a considerably more efficient algorithm of complexity of order $p^3$.
\end{rem}

\section{Periodic orbits on the interval determine the left-right behaviour}
\label{sec: periodic points on the interval determine left-right behaviour}
In this section and throughout the article, $\I$ always denotes an interval in $\R$; it may be open, closed, or half-open.
By interval, we always mean a \emph{non-degenerate} interval unless explicitly stated otherwise.
If we wish to specify the left and right endpoints of $\I$, we may write $\I=\llbracket \0,\1\rrbracket$, where the left bracket $\llbracket$ is either $($ or $[$, the right bracket $\rrbracket$ is either $)$ or $]$, $\0\in\R\cup \{-\infty\}$, and $\1\in \R\cup \{\infty\}$ with $\1>\0$.
By an endomorphism on $\I$, we mean a continuous self-map on $\I$.
At times, we may refer to properties of such an endomorphism at the boundary points $\0$ or $\1$.
In such a case, if the respective boundary point is not an element of $\I$, for example if $\I=(-\infty,\infty)$, then the corresponding statements are understood to be vacuously true.

In Section~\ref{sec: examples}, we provide examples showing that the collection of periodic orbits need not determine an endomorphism on $\I$, even under the assumption that periodic points are dense.
That is, we give distinct continuous maps with the same dense collection of periodic orbits.
In those examples, however, each periodic point has the same left-right behaviour under the respective maps.
In this section, we prove that this is a general fact:
Theorem~\ref{thm: periodic orbits determine left-right behaviour} below shows that the left-right behaviour of each periodic point is fully determined by the collection of periodic orbits. 
Together with Lemma~\ref{lem: left-right determines map} from the previous section, this shows that two endomorphisms on $\I$ with the same periodic orbits must coincide whenever prime periodic points are dense; 
note that this is, in particular, the case for topologically mixing maps on $\I$, see Corollary~\ref{cor: mixing maps are fully determined by periodic orbits}.

To illustrate the idea behind 
Theorem~\ref{thm: periodic orbits determine left-right behaviour}, let us briefly restrict attention to  maps $f$ on the unit interval $\I=[0,1]$ with dense periodic points, only finitely many periodic orbits of any given period, and all periodic points \emph{stable} under perturbations: that is, we assume that for each $p\in \N$, every $p$-periodic point corresponds to a crossing of the diagonal by the graph of $f^p$.

To see that, in this situation, the left-right behaviour is determined by the periodic orbits, first note that the assumptions of finiteness and denseness together imply that, if $0$ is a fixed point of $f^p$, then the graph of $f^p$ must lie above the diagonal to the right of $0$; otherwise, $0$ would either be accumulated by fixed points of $f^p$---violating the finiteness assumption---or $0$ would be attracting for $f^p$, contradicting the denseness of periodic points.
As a consequence, determining the left-right behaviour of a $p$-periodic point $x$ of $f$, for any $p>2$, amounts to counting the crossings of the diagonal by the graph of $f^q$ to the left of $x$ for each $q<p$: if the number of such crossings is even, then $x$ is mapped to its right under $f^q$ (as its graph lies above the diagonal at position $x$), and to its left otherwise.  
However, each crossing corresponds to a periodic point and, under 
the stability assumption, every periodic point gives rise to a 
crossing. 
Consequently, the left-right behaviour of any periodic point is 
entirely determined by the positions of the periodic points of 
smaller period.

Note also that if \( f \) has dense prime points, then combining the algorithm from Lemma~\ref{lem: left-right determines map} with the preceding discussion yields an algorithm for approximating \( f \) based on the knowledge of its periodic orbits.

While the above captures the situation for generic maps, the general case is more involved. 
First, the possible existence of infinitely many fixed points 
precludes a simple counting of crossings even if $\I=[0,1]$. 
Furthermore, periodic points need not all be stable in general, so the one-to-one correspondence between crossings and periodic points may fail.
To address these complications, we need a careful analysis of the local behaviour of $f^{ip}$ around all $p$-periodic points for all $i,p\in \N$.
To that end, we to introduce some notation and terminology.

Given an endomorphism $f$ on $\I$, we write $\Fix_p(f)$ for the collection of fixed points of $f^p$.
Note that we call $x\in \I$ \emph{$p$-periodic (under $f$)} if $x\in \Fix_p(f)$ and $p$ is minimal with this property.
In the following, by \emph{neighbourhood} of a point $x\in \I$ we refer to an \emph{interval} that contains $x$ in its interior.

First, consider $x\in \d\Fix_p(f)\setminus \{\0,\1\}$.
We say $x$ is of \emph{type $(+,+)$ (under $f^p$)} if
$f^p(y)\geq y$ on a neighbourhood of $x$; we may call such neighbourhood also a $(+,+)$-neighbourhood of $x$ (under $f^{p}$); see Figure \ref{fig:local_behaviour_plus_plus} for an illustration of corresponding local behaviour. 
We say $x$ is of type $(+,-)$ under $f^p$ if there is a neighbourhood $U$ of $x$ such that
$f^p(y)\geq y$ for all $y\in U$ to the left of $x$ and $f^p(y)\leq y$ for all $y\in U$ to the right of $x$ and such that there are $y_-<x<y_+$ in $U$ with $f^p(y_-)>y_-$ and $f^p(y_+)<y_+$; we may call the respective set $U$ a $(+,-)$-neighbourhood of $x$ (under $f^p$).
The types $(-,-)$ and $(-,+)$  (and the corresponding neighbourhoods) are defined analogously.

It is noteworthy that we explicitly do not assume the respective
properties to be satisfied on sufficiently small neighbourhoods; for only that way, the above terminology also applies meaningfully to boundary points of $\mathring\Fix_p(f)$.
Note further that the above types are not mutually exclusive: for example, $x$ may admit a $(+,-)$-neighbourhood with a proper subset that serves as a $(+,+)$-neighbourhood, see for instance the left sketch in Figure \ref{fig:local_behaviour_plus_minus_plus_pm}.
However, $x$ cannot simultaneously be of type $(+,+)$ and of type $(-,-)$ under $f^p$ since $x$ is in the boundary of $\Fix_p(f)$.

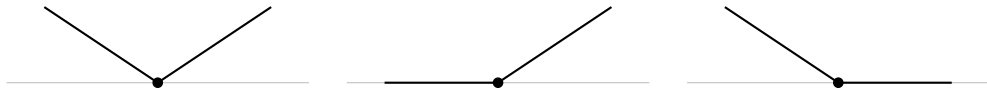
\begin{figure}[h]
    \begin{tikzpicture}
        \begin{scope}
            \draw[lightgray] (-2,0) -- (2,0);
            
            \draw[thick] (-1.5,1) -- (0,0) -- (1.5,1);
            
            \fill (0,0) circle (2pt);
        \end{scope}
        
        \begin{scope}[xshift=4.5cm]
            \draw[lightgray] (-2,0) -- (2,0);
            
            \draw[thick] (-1.5,0) -- (0,0) -- (1.5,1);
            
            \fill (0,0) circle (2pt);
        \end{scope}
        
        \begin{scope}[xshift=9.0cm]
            \draw[lightgray] (-2,0) -- (2,0);
            
            \draw[thick] (-1.5,1) -- (0,0) -- (1.5,0);
            
            \fill (0,0) circle (2pt);
        \end{scope}
    \end{tikzpicture}
    \caption{
        Possible local behaviour of the map $f^p-\Id$ on a $(+,+)$-neighbourhood of a fixed point in $\d\Fix_p(f)\setminus \{\0,\1\}$ (indicated as a dot).
    }
    \label{fig:local_behaviour_plus_plus}
\end{figure}

Still assuming that $x\in \partial \Fix_p(f)\setminus \{\0,\1\}$, we say $x$ is of type $(+,\pm)$ under $f^p$ if it is neither of type $(+,+)$ nor $(+,-)$ and if there is
a neighbourhood $U$ of $x$ with $f^p(y)\geq y$ for all $y\in U$ to the left of $x$ and such that there are $y_-<x<y_+^+,y_+^-$ in $U$ with $f^p(y_-)>y_-$, $f^p(y_+^+)>y_+^+$ and $f^p(y_+^-)<y_+^-$; we may call $U$ a $(+,\pm)$-neighbourhood of $x$ (under $f^p$).
Note that in this situation, the points $y_+^+$ and $y_+^-$ can be chosen arbitrarily close to $x$.
The types $(\pm,+)$, $(-,\pm)$, and $(\pm,-)$ as well as the corresponding neighbourhoods are defined analogously.

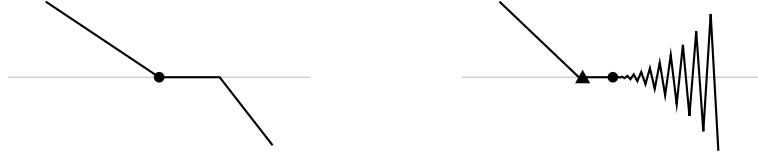
\begin{figure}[h]
    \begin{tikzpicture}
        \begin{scope}
            \draw[lightgray] (-2,0) -- (2,0);
            
            \draw[thick] (-1.5,1) -- (0,0) -- (0.8,0) -- (1.5,-0.9);
            
            \fill (0,0) circle (2pt);
        \end{scope}
        
        \begin{scope}[xshift=6cm]
            \draw[lightgray] (-2,0) -- (2,0);
            
            \draw[thick] (-1.5,1.0) -- (-0.45,0);
            \fill[black] (-0.4,0.1) -- (-0.5,-0.08) -- (-0.3,-0.08) -- cycle;
            \draw[thick] (-0.4,0) -- (0,0);
            
            \draw[thick] (0,0)  -- (0.001913, -0.0000)  -- (0.007653, 0.0000)  -- (0.017219, -0.0001)  -- (0.030612, 0.0005)  -- (0.047832, -0.0011)  -- (0.068878, 0.0024)  -- (0.093750, -0.0044)  -- (0.122449, 0.0075)  -- (0.154974, -0.0120)  -- (0.191327, 0.0183)  -- (0.231505, -0.0268)  -- (0.275510, 0.0380)  -- (0.323342, -0.0523)  -- (0.375000, 0.0703)  -- (0.430485, -0.0927)  -- (0.489796, 0.1200)  -- (0.552934, -0.1529)  -- (0.619898, 0.1921)  -- (0.690689, -0.2385)  -- (0.765306, 0.2928)  -- (0.843750, -0.3560)  -- (0.926020, 0.4288)  -- (1.012117, -0.5122)  -- (1.102041, 0.6072)  -- (1.195791, -0.7150)  -- (1.293367, 0.8364)  -- (1.394770, -0.9727);
            
            \fill (0,0) circle (2pt);
        \end{scope}
    \end{tikzpicture}
    
    \caption{Possible local behaviour of $f^p - \Id$ on a $(+,-)$-neighbourhood (left) and a $(+,\pm)$-neighbourhood (right) of a fixed point in $\partial \Fix_p(f) \setminus \{\0,\1\}$ (indicated by a dot). 
    Note that the marked fixed point in the left sketch also admits a $(+,+)$-neighbourhood; the second fixed point in the right sketch (indicated by a triangle) is of type $(+,+)$ and of no other type.
}
    \label{fig:local_behaviour_plus_minus_plus_pm}    
\end{figure}

Finally, we say $x\in \partial\Fix_p(f)\setminus \{\0,\1\}$ is of type $(\pm,\pm)$ if it is of none of the above kinds.
We remark that, if $x$ is of type $(\pm,\pm)$, it may happen that
$\llbracket\0,x]\ssq \Fix_p(f)$ (or $[x,\1\rrbracket\ssq \Fix_p(f)$); see Figure~\ref{fig:local_behaviour_pm_pm}.
Although this is slightly at odds with the notation, it does not affect the subsequent discussion.
Similarly, it may happen that $x$ is accumulated from, say, the left by points where the graph of $f^p$ lies above the diagonal and by points where it lies below the diagonal, while also being the left endpoint of an interval in $\Fix_p(f)$ whose right endpoint $y$ is accumulated from the right by points where the graph of $f^p$ lies above the diagonal and by points where it lies below the diagonal (in that case, $y$ is also of type $(\pm,\pm)$).
If $x$ is of type $(\pm,\pm)$ and neither of these alternatives holds, however, then $x$ is accumulated from both sides by points where $f^p$ lies strictly above and strictly below the diagonal.

\begin{figure}[h]
    \begin{tikzpicture}
        \begin{scope}
            \draw[lightgray] (-2,0) -- (2,0);

            \draw[thick] (0,0)  -- (-0.001913, -0.0000)  -- (-0.007653, 0.0000)  -- (-0.017219, -0.0001)  -- (-0.030612, 0.0005)  -- (-0.047832, -0.0011)  -- (-0.068878, 0.0024)  -- (-0.093750, -0.0044)  -- (-0.122449, 0.0075)  -- (-0.154974, -0.0120)  -- (-0.191327, 0.0183)  -- (-0.231505, -0.0268)  -- (-0.275510, 0.0380)  -- (-0.323342, -0.0523)  -- (-0.375000, 0.0703)  -- (-0.430485, -0.0927)  -- (-0.489796, 0.1200)  -- (-0.552934, -0.1529)  -- (-0.619898, 0.1921)  -- (-0.690689, -0.2385)  -- (-0.765306, 0.2928)  -- (-0.843750, -0.3560)  -- (-0.926020, 0.4288)  -- (-1.012117, -0.5122)  -- (-1.102041, 0.6072)  -- (-1.195791, -0.7150)  -- (-1.293367, 0.8364)  -- (-1.394770, -0.9727);

            \draw[thick] (0,0)  -- (0.001913, -0.0000)  -- (0.007653, 0.0000)  -- (0.017219, -0.0001)  -- (0.030612, 0.0005)  -- (0.047832, -0.0011)  -- (0.068878, 0.0024)  -- (0.093750, -0.0044)  -- (0.122449, 0.0075)  -- (0.154974, -0.0120)  -- (0.191327, 0.0183)  -- (0.231505, -0.0268)  -- (0.275510, 0.0380)  -- (0.323342, -0.0523)  -- (0.375000, 0.0703)  -- (0.430485, -0.0927)  -- (0.489796, 0.1200)  -- (0.552934, -0.1529)  -- (0.619898, 0.1921)  -- (0.690689, -0.2385)  -- (0.765306, 0.2928)  -- (0.843750, -0.3560)  -- (0.926020, 0.4288)  -- (1.012117, -0.5122)  -- (1.102041, 0.6072)  -- (1.195791, -0.7150)  -- (1.293367, 0.8364)  -- (1.394770, -0.9727);
            
            \fill (0,0) circle (2pt);
        \end{scope}

        \begin{scope}[xshift=4.5cm]
            \draw[lightgray] (-2,0) -- (2,0);

            \draw[thick] (-0.250000,0)  -- (-0.251594, -0.0000)  -- (-0.256378, 0.0000)  -- (-0.264349, -0.0002)  -- (-0.275510, 0.0005)  -- (-0.289860, -0.0012)  -- (-0.307398, 0.0024)  -- (-0.328125, -0.0045)  -- (-0.352041, 0.0077)  -- (-0.379145, -0.0123)  -- (-0.409439, 0.0188)  -- (-0.442921, -0.0275)  -- (-0.479592, 0.0390)  -- (-0.519452, -0.0537)  -- (-0.562500, 0.0723)  -- (-0.608737, -0.0952)  -- (-0.658163, 0.1233)  -- (-0.710778, -0.1571)  -- (-0.766582, 0.1975)  -- (-0.825574, -0.2452)  -- (-0.887755, 0.3010)  -- (-0.953125, -0.3658)  -- (-1.021684, 0.4407)  -- (-1.093431, -0.5264)  -- (-1.168367, 0.6241)  -- (-1.246492, -0.7348)  -- (-1.327806, 0.8596)  -- (-1.412309, -0.9997);
            
            \draw[thick] (0.250000,0)  -- (0.251594, -0.0000)  -- (0.256378, 0.0000)  -- (0.264349, -0.0002)  -- (0.275510, 0.0005)  -- (0.289860, -0.0012)  -- (0.307398, 0.0024)  -- (0.328125, -0.0045)  -- (0.352041, 0.0077)  -- (0.379145, -0.0123)  -- (0.409439, 0.0188)  -- (0.442921, -0.0275)  -- (0.479592, 0.0390)  -- (0.519452, -0.0537)  -- (0.562500, 0.0723)  -- (0.608737, -0.0952)  -- (0.658163, 0.1233)  -- (0.710778, -0.1571)  -- (0.766582, 0.1975)  -- (0.825574, -0.2452)  -- (0.887755, 0.3010)  -- (0.953125, -0.3658)  -- (1.021684, 0.4407)  -- (1.093431, -0.5264)  -- (1.168367, 0.6241)  -- (1.246492, -0.7348)  -- (1.327806, 0.8596)  -- (1.412309, -0.9997);
            
            \draw[thick] (-0.25,0)  -- (0.25,0);

            \fill (-0.25,0) circle (2pt);
            \fill (0.25,0) circle (2pt);
        \end{scope}

        \begin{scope}[xshift=9cm]
            \draw[lightgray] (-2,0) -- (2,0);

            \draw[thick] (1.5,0) -- (0,0)  -- (-0.001913, -0.0000)  -- (-0.007653, 0.0000)  -- (-0.017219, -0.0001)  -- (-0.030612, 0.0005)  -- (-0.047832, -0.0011)  -- (-0.068878, 0.0024)  -- (-0.093750, -0.0044)  -- (-0.122449, 0.0075)  -- (-0.154974, -0.0120)  -- (-0.191327, 0.0183)  -- (-0.231505, -0.0268)  -- (-0.275510, 0.0380)  -- (-0.323342, -0.0523)  -- (-0.375000, 0.0703)  -- (-0.430485, -0.0927)  -- (-0.489796, 0.1200)  -- (-0.552934, -0.1529)  -- (-0.619898, 0.1921)  -- (-0.690689, -0.2385)  -- (-0.765306, 0.2928)  -- (-0.843750, -0.3560)  -- (-0.926020, 0.4288)  -- (-1.012117, -0.5122)  -- (-1.102041, 0.6072)  -- (-1.195791, -0.7150)  -- (-1.293367, 0.8364)  -- (-1.394770, -0.9727);
    
            \fill (0,0) circle (2pt);
    
            \draw (1.5,-0.1) -- (1.5,0.1) node[below=5pt]{$\1$};
        \end{scope}
    
    \end{tikzpicture}
    
    \caption{Possible local behaviour of the map $f^p-\Id$ on a $(\pm,\pm)$-neighbourhood of a fixed point in $\d\Fix_p(f)\setminus \{\0,\1\}$ (indicated by a dot).}
    \label{fig:local_behaviour_pm_pm}    
\end{figure}
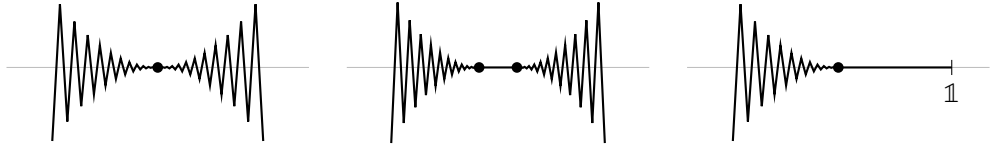

The above types cover $\partial \Fix_p(f)\setminus \{\0,\1\}$ but, as noted before, are not mutually exclusive.
For brevity, we may refer to $x$ as being of type $(\ast,+)$
if $x$ is of any of the types $(+,+)$, $(-,+)$ or $(\pm,+)$ (where only the second coordinate is determined).
The types $(\ast,-)$, $(\ast,\pm)$,
$(+,\ast)$, $(-,\ast)$ and $(\pm,\ast)$ are defined accordingly.

Now, consider $x\in \mathring\Fix_p(f)$, and let $\llbracket a,b\rrbracket$ be the connected component of $\Fix_p(f)$ containing $x$. 
We say $x$ is of type $(+,+)^l$ if $a>\0$ 
and $a$ is of type $(+,+)$.
Similarly, we 
say $x$ is of type $(+,+)^r$ if $b<\1$ 
and $b$ is of type $(+,+)$.
The types $(-,+)^r$, $(-,\pm)^r$, 
$(-,+)^l$, $(-,\pm)^l$, etc.\ are defined analogously. 
As with the types introduced for $\partial\Fix_p(f)\setminus\{\0,\1\}$, the above covers $\mathring\Fix_p(f)\setminus\{\0,\1\}$---assuming $f^p\neq \Id$---but the types are not mutually exclusive.

To understand the left-right behaviour of points under $f$, we must determine whether the graphs of the iterates $f^p$ lie above or below the diagonal near fixed points of $f^p$, and what this implies for higher powers $f^{ip}$ of these iterates.
To that end, we need the following statement, where items \ref{item: types of fixed points 1}--\ref{item: types of fixed points 3} address the preservation of types, while the remaining parts describe how the type of a point \( x \in \partial\Fix_p(f) \) under \( f^p \) may differ from its type under \( f^{ip} \) for \( i > 1 \).
The proof is, in principle, straightforward but tedious and would detract from the 
main ideas behind Theorem~\ref{thm: periodic orbits determine left-right behaviour}; 
we therefore defer it to the final section.

\begin{lem}\label{lem: local type of fixed points}
 Let $f\: \I\to \I$ be continuous and consider $x\in \Fix_p(f)\setminus \{\0,\1\}$ with $p\in\N$.
 \begin{enumerate}[label=(\roman*),leftmargin=*]
  \item
  If $x$ is of type $(\alpha,\beta)^r$ [resp.\ $(\alpha,\beta)^l$] under $f^p$ for some $\alpha,\beta \in \{+,-,\pm\}$, then $x$ is of type $(\alpha,\beta)^r$ [resp.\ $(\alpha,\beta)^l$] under $f^{ip}$ for every $i\in \N$.
  \label{item: types of fixed points 1}
  \item If $x$ is of type $(-,-)$  [resp.\ $(+,+)$] under $f^p$, then $x$ is of type $(-,-)$ [resp.\ $(+,+)$] under $f^{ip}$ for every $i\in \N$. \label{item: types of fixed points 2}
 \item If $x$ is of type $(\ast,+)$ [resp.\ $(-,\ast)$] under $f^p$, then $x$ is of type $(\ast,+)$ [resp.\ $(-,\ast)$] under $f^{ip}$ for every $i\in \N$. \label{item: types of fixed points 3}
\item
   If $x$ is of type $(\pm,-)$, then $x$ is of type $(\pm,-)$ or $(\pm,\pm)$ under $f^{ip}$ for every $i\in \N$.
   \label{item: types of fixed points 4}
   \item  If $x$ is of type $(+,\pm)$, then $x$ is of type $(+,\pm)$ or $(\pm,\pm)$ under $f^{ip}$ for every $i\in \N$.
   \label{item: types of fixed points 5}
  \item
  If $x$ is of type $(+,-)$, then either\label{item: types of fixed points 6}
   \begin{itemize}
    \item[(a)] $x$ is of type $(\ast,-)$ under $f^{ip}$ for every $i\in \N$ or
    \item[(b)] $x$ is of type $(\ast, -)$ under $f^{ip}$ for every odd $i\in \N$ and of type $(\ast, +)$ under $f^{ip}$ for every even $i\in \N$ or 
    \item[(c)] $x$ is
     of type $(\ast, -)$ or $(\ast,\pm)$ under $f^{ip}$ for every odd $i\in \N$ and
    of type $(\ast,\pm)$ under $f^{ip}$ for every even $i\in \N$ or
    \item[(d)] 
    $x$ is of type $(\ast, -)$ under $f^{ip}$ for every odd $i\in \N$ and $x\in \mathring\Fix_{i p}(f)$ for every even $i\in \N$.
    \end{itemize}
    \item
  If $x$ is of type $(+,-)$, then---focusing on the first coordinate---we analogously get that either\label{item: types of fixed points 7}
   \begin{itemize}
    \item[(a)] $x$ is of type $(+,\ast)$ under $f^{ip}$ for every $i\in \N$ or
    \item[(b)] $x$ is of type $(+, \ast)$ under $f^{ip}$ for every odd $i\in \N$ and of type $(-,\ast)$ under $f^{ip}$ for every even $i\in \N$ or
    \item[(c)] $x$ is
     of type $(+,\ast)$ or $(\pm,\ast)$ under $f^{ip}$ for every odd $i\in \N$ and
    of type $(\pm,\ast)$ under $f^{ip}$ for every even $i\in \N$ or
    \item[(d)] $x$ is of type $(+,\ast)$ under $f^{ip}$ for every odd $i\in \N$ and $x\in \mathring\Fix_{i p}(f)$ for every even $i\in \N$.
   \end{itemize}
 \end{enumerate}
\end{lem}
\begin{proof}
See Section~\ref{sec: proof of technical lemma}.
\end{proof}

We call $x\in \d \Fix_p(f)\setminus \{\0,\1\}$ a \emph{stable} fixed point of $f^p$ if $x$ is of some type other than $(+,+)$ or $(-,-)$ under $f^p$ with respect to some neighbourhood (regardless of whether $x$ also admits a $(+,+)$- or $(-,-)$-neighbourhood).
We write $\Fix_p^s(f)$ for the collection of all stable fixed points of $f^p$.
Here, \emph{stability} refers to stability under perturbations: if a fixed point $x$ of $f^p$ is stable in the above sense and $g$ is sufficiently close to $f$ (in the $C^0$-topology), then $g^p$ also has a fixed point close to the connected component of $\Fix_p(f)$ containing $x$.
Note that every interval on which $f^p$ crosses the diagonal, that is, on which $f^p$ takes values both strictly above and strictly below the diagonal, necessarily intersects $\Fix_p^s(f)$.

We call a $p$-periodic point $x$ a \emph{stable $p$-periodic point} if $x\in \Fix_p^s(f)$. 
Note that if $x\in \Fix_p^s(f)$ is $q$-periodic (for some $q\mid p$), then also $x\in \Fix_q^s(f)$, that is, $x$ is a stable $q$-periodic point.
We write $\PER_p(f)$ [$\PER_p^s(f)$] for the collection of all [stable] $p$-periodic \emph{orbits}\footnote{We call a $p$-periodic orbit $\mc O(x)$ \emph{stable} if there is $x'\in \mc O(x)$ with $x'\in \Fix_p^s(f)$.}
of $f$
and $\PER(f)=\bigcup_{p\geq 1} \PER_p(f)$ as well as
$\PER^s(f)=\bigcup_{p\geq 1} \PER^s_p(f)$.
We write $\Per_p(f)$ [$\Per_p^s(f)$] for the collection of [stable] $p$-periodic points of $f$.

Given two endomorphisms $f$ and $g$ on $\I$,
we set
\[
 U_n=U_n(f,g)=\{x\in \I\: (f^n(x)-x)\cdot (g^n(x)-x)<0\}.
\]
In other words, $U_n$ is the collection of points whose left-right behaviour under $f$ and $g$ strictly differs at time $n$, in the sense that $f^n$ and $g^n$ map each such point to opposite sides of itself.
\begin{fact}\label{fact: leftmost endpoints not in Un-component (0,b)}
 Note that $U_n$ cannot contain
 (for any $n\in\N$) the leftmost point $x_1$ of a periodic orbit $x_1< \ldots <x_p$ in $\PER(f)\cap \PER(g)$ since otherwise, $f^n(x_1)<x_1$ or $g^n(x_1)<x_1$ in contradiction to $x_1$ being leftmost.
An analogous observation holds for the rightmost point $x_p$.
\end{fact}

Given some $x\in U_n$, let
$a<b\in \Fix_n(f)\cap \Fix_n(g)$ be the closest points to the left and right of $x$, respectively, with $\mc O_f(a)=\mc O_g(a)$ (or $a=\0$ if no such point exists) and $\mc O_f(b)=\mc O_g(b)$ (or $b=\1$ if no such point exists).
Note that $a=\0$ only if there is no $y>\0$ in $\Fix_n(f)\cap \Fix_n(g)$ with $\mc O_f(y)=\mc O_g(y)$.
Indeed, if such $y$ exists, then the collections of orbits of points in $\Fix_n(f)\cap [y,x]$ and $\Fix_n(g)\cap [y,x]$ are compact in the Hausdorff metric, and hence so is their intersection.
It follows that the closest point to the left of $x$ satisfying the above assumptions exists.
The analogous statement holds for $b$.

Given $a$ and $b$ as above, we call the interval $(a,b)$ the \emph{$U_n$-component} of $x$.
Observe that a priori, we have not specified the period of $a$ or $b$ (under $f$ and $g$, respectively); it may well be a proper divisor of $n$.
 Moreover, note that the $U_n$-component $(a,b)$ is not necessarily a subset of $U_n$ as it may intersect orbits from $\bigcup_{p|n}\PER_p(f) \triangle \PER_p(g)$.

\begin{fact}\label{fact: no crossings in Un-component}
With $(a,b)$ a $U_n$-component of $x$ as above,
if $\PER^s(f)\ssq \PER(g)$ and $\PER^s(g)\ssq \PER(f)$,
then $f^p$ and $g^p$ do not cross the diagonal on $(a,b)$ for any $p \mid n$ (indeed, any crossing would yield a stable, hence shared, periodic orbit intersecting $(a,b)$, which would make the $U_n$-component of $x$ properly contained in $(a,b)$, contrary to the assumptions).
In particular, we have that
$f^n|_{(a,b)}$ lies on or above the diagonal and $g^n|_{(a,b)}$ lies on or below the diagonal (or vice versa).
\end{fact}

\begin{fact}\label{fact: same type in Un-component}
As before, suppose $\PER^s(f)\ssq \PER(g)$ and $\PER^s(g)\ssq \PER(f)$,
and let $(a,b)$ be a $U_n$-component.
If $a>\0$, then Fact~\ref{fact: no crossings in Un-component} implies that $a$ must be of type $(\ast,+)$, $(\ast,-)$, $(\ast,+)^r$, or $(\ast,-)^r$ under each of $f^n$ and $g^n$, though the type necessarily differs between the two maps.
Indeed, if $a$ was of type $(\ast,\pm)$ or $(\ast,\pm)^r$ under $f^n$, $f^n$ would cross the diagonal in $(a,b)$.
Likewise, if $b<\1$, then $b$ must 
be of type $(+,\ast)$, $(-,\ast)$, $(+,\ast)^l$, or $(-,\ast)^l$ under each of $f^n$ and $g^n$, though again the type differs. 
More precisely, since $(a,b)\cap U_n\neq\emptyset$, if $a$ is of 
type $(\ast,+)$ or $(\ast,+)^r$ under $f^n$, then it is of type $(\ast,-)$ or $(\ast,-)^r$ under $g^n$, and vice versa; similarly, if $b$ is of type $(+,\ast)$ or $(+,\ast)^l$ under $f^n$, then it is of type $(-,\ast)$ or $(-,\ast)^l$ under $g^n$.
\end{fact}

\begin{fact}\label{fact: different types under different iterates}
In the situation and with the notation of 
Fact~\ref{fact: same type in Un-component}, if $p\mid n$ and $a$ 
is $p$-periodic under $f$, the second coordinate of the type of $a$ under $f^n$ may differ from that under $f^p$; that is, for example, $a$ may be of type $(\ast,+)$ under $f^n$ but not under $f^p$.
Due to Fact~\ref{fact: no crossings in Un-component} and by inspection of Lemma~\ref{lem: local type of fixed points}, we have that this can only occur if $a$ is of type $(+,-)$ under $f^p$. 
Indeed, Fact~\ref{fact: no crossings in Un-component} rules out type $(\ast,\pm)$ under $f^p$ and $f^n$ alike (just as discussed in Fact~\ref{fact: same type in Un-component});
further, if $a$ were of type $(\ast,+)$ 
or $(\alpha,\beta)^r$ (for any $\alpha$ and $\beta$) under $f^p$, the same would hold under 
$f^n$ by Lemma~\ref{lem: local type of fixed points} 
\ref{item: types of fixed points 1} \&
\ref{item: types of fixed points 3}.
\newline
\indent
Note that Lemma~\ref{lem: local type of fixed points}---specifically, item \ref{item: types of fixed points 6}---further gives that in the above situation, $a$ is necessarily of type $(\ast,-)$ under $f^{ip}$ for all odd $i\in \N$.
More precisely, alternative~(c) in item~\ref{item: types of fixed points 6} is ruled out by Fact~\ref{fact: no crossings in Un-component}. 
Thus, among the remaining alternatives, only~(b) and~(d) are compatible with $a$ having types with different second coordinates under $f^p$ and $f^n$.
\end{fact}

\begin{fact}\label{fact: Un-components dont contain endpoints}
 Assume again that $\PER^s(f)\ssq \PER(g)$ and $\PER^s(g)\ssq \PER(f)$ and consider $p\in\N$ with $p\nmid n$; this includes in particular the case $p>n$.
 Suppose 
 $\mc O_f(x)=\mc O_g(x) \in \PER_p(f)\cap \PER_p(g)$ for some
 $x$ in a $U_n$-component (of some point $y\in U_n$). 
 Then Fact~\ref{fact: no crossings in Un-component} implies that
$x$ necessarily lies in $U_n$ itself---note that $f^n(x)-x\neq 0$ and $g^n(x)-x\neq 0$ since $p\nmid n$.
\newline
\indent
 As a particular consequence, $U_n$-components of the form $(\0,b)$ cannot intersect orbits from $\PER_p(f)\cap \PER_p(g)$ (whenever $p\nmid n$):
First, Fact~\ref{fact: leftmost endpoints not in Un-component (0,b)} 
implies that if an orbit in $\PER_p(f)\cap \PER_p(g)$ meets a 
$U_n$-component of the form $(\0,b)$, then its leftmost point must 
lie outside $(\0,b)$, hence it must be $\0$. 
In other 
words, we would have $\mc O_f(\0)=\mc O_g(\0)\in \PER_p(f)\cap 
\PER_p(g)$.
Now, since $p\nmid n$, this would imply 
$f^n(\0),g^n(\0)>\0$ and hence, by continuity, $f^n(z),g^n(z)>\0$ for 
sufficiently small $z\in (\0,b)$, contradicting 
Fact~\ref{fact: no crossings in Un-component}.
A similar statement holds for $U_n$-components of the form $(a,\1)$.
\end{fact}

\begin{thm}\label{thm: periodic orbits determine left-right behaviour}
 Suppose $f,g\: \I\to \I$ are continuous with $\PER(f)^s\ssq \PER(g)$ and $\PER(g)^s\ssq \PER(f)$.
 Suppose there is $n\in\N$ such that $U_n\neq \emptyset$.

 Then $U_n$ does not intersect orbits from $\PER(f)\cap \PER(g)$.
\end{thm}
\begin{proof}
We write $\Per_p(f,g)\ssq \I$ for the collection of points whose $f$-orbits and $g$-orbits, respectively, are in $\PER_p(f)\cap \PER_p(g)$.
Our goal is to show that $U_n$ and $\Per_p(f,g)$ are disjoint for all $n, p \in \N$.
For a contradiction, assume that $n$ is minimal with the property that $U_n\cap \Per_p(f,g)\neq\emptyset$ for some $p$.

First, assume for a contradiction that there is $p\leq n$ and $x\in \Per_p(f,g) \cap U_n$.
Note that $n\bmod p<n$, and hence, by minimality of $n$,
$f^{n\bmod p}(x)>x$ if and only if $g^{n\bmod p}(x)>x$,
and $f^{n\bmod p}(x)<x$ if and only if $g^{n\bmod p}(x)<x$.
This contradicts the assumption that $x \in U_n$, since $f^n(x) = f^{n \bmod p}(x)$ and $g^n(x) = g^{n \bmod p}(x)$.

Therefore, in what follows, we assume for a contradiction that there exists $P > n$ and a point $x_0 \in \Per_P(f,g) \cap U_n$; let $P$ be minimal with this property.
We will show that this implies
\begin{align}\label{eq: contradiction to same number of left visits}
| \#\{1\leq m<P\:f^m(x_0)>x_0\}-\#\{1\leq m<P\:g^m(x_0)>x_0\}|>0
\end{align}
which contradicts the assumption that $\mc O_f(x_0)=\mc O_g(x_0)$:
if $\mc O_f(x_0)=\mc O_g(x_0)$, then regardless of the 
combinatorics of the orbit, $x_0$ necessarily visits the same number of points to its right under $f$ as under $g$ before returning to itself.

Towards proving \eqref{eq: contradiction to same number of left visits}, we first show that the left-right behaviour of $x_0$ under $f$ and $g$ can differ only at times divisible by $n$, see \eqref{eq: x0 not in Up for p not multiple of n} below. 
Assuming for concreteness that $f^n(x_0)>x_0$ and $g^n(x_0)<x_0$, we will then show that $f^m(x_0)>x_0$ for every $m$ with $n\mid m$.
Together with \eqref{eq: x0 not in Up for p not multiple of n}, this implies that $x_0$ is mapped to the right under iteration of $f$ more often than under iteration of $g$, which proves \eqref{eq: contradiction to same number of left visits}.

We will need the following 

\begin{claim}\label{claim: boundary points of Un components are n-periodic}
  Let $(a,b)$ be the $U_n$-component of $x_0$.
  Then $\0<a<b<\1$; further, $a$ or $b$ is in $\Per_n(f,g)$.
 \end{claim}
\begin{proof}[Proof of Claim~\ref{claim: boundary points of Un components are n-periodic}]
Observe that $\0<a<b<\1$ due to the second part of Fact~\ref{fact: Un-components dont contain endpoints} (recall that $x_0\in \Per_P(f,g)$ with $P>n$).
Hence, by definition of $U_n$-components, $a\in \Per_q(f,g)$ and $b\in \Per_{q'}(f,g)$ for some divisors $q$ and $q'$ of $n$.
To prove the claim, assume for a contradiction that $q,q'<n$.

 By minimality of $n$, $U_q\cap \Per_P(f,g)=\emptyset$ so that $f^q(x_0)$ and $g^q(x_0)$ are simultaneously greater or less than $x_0$---note that $f^q(x_0),g^q(x_0)\neq x_0$ as $P>q$.
 In the first case, $a$ is of type $(\ast,+)$ or $(\ast,+)^r$ under $f^q$ and $g^q$; in the second case,
 $a$ is of type $(\ast,-)$ or $(\ast,-)^r$ under $f^q$ and $g^q$---in both cases, we use that by Fact~\ref{fact: no crossings in Un-component}, neither $f^q$ nor $g^q$ crosses the diagonal in $(a,b)$.
 An analogous statement holds for $b$ under $f^{q'}$ and $g^{q'}$:
either $f^{q'}(x_0)$ and $g^{q'}(x_0)$ are both greater than $x_0$, in which case $b$ is of type $(+,\ast)$ or $(+,\ast)^l$ under both $f^{q'}$ and $g^{q'}$; or else $b$ is of type $(-,\ast)$ or $(-,\ast)^l$ under both maps.

 If $q=q'$, the above boils down to saying that either
 $a$ is of type $(\ast,+)$ or $(\ast,+)^r$ under $f^q$ and $g^q$ or $b$ is of type $(-,\ast)$ or $(-,\ast)^l$ under $f^{q}$ and $g^{q}$. 
 It follows from Lemma~\ref{lem: local type of fixed points} \ref{item: types of fixed points 1} and \ref{item: types of fixed points 3} that, in the first case, $a$ is of type $(\ast,+)$ or $(\ast,+)^r$ under $f^n$ and $g^n$, whilst in the second case $b$ is of type $(-,\ast)$ or $(-,\ast)^l$ under $f^n$ and $g^n$.
 In either case, this contradicts the last part of Fact~\ref{fact: same type in Un-component}.

We may therefore assume without loss of generality that $q < q'$.
Since $a$ is the left endpoint of a $U_n$-component, the second coordinates of its type under $f^n$ and $g^n$ must be of opposite sign (Fact~\ref{fact: same type in Un-component}).
On the other hand, as argued above (in the second paragraph of this proof), $a$ is either of type $(\ast,+)$ or $(\ast,+)^r$ under both $f^q$ and $g^q$, or of type $(\ast,-)$ or $(\ast,-)^r$ under both.
By Fact~\ref{fact: different types under different iterates}, $a$ must therefore be of type $(+,-)$ under at least one of $f^q$ and $g^q$---while the types may differ between the two maps, they are in any case of the form $(\ast,-)$ or $(\ast,-)^r$.
Fact~\ref{fact: different types under different iterates} further implies that $2q \mid n$ (since for odd $i$, $a$ is of type $(\ast,-)$ or $(\ast,-)^r$ under both $f^{iq}$ and $g^{iq}$), and that under one of $f^{2q}$ and $g^{2q}$, $a$ is of type $(\ast,-)$ or $(\ast,-)^r$, while under the other it is of type $(\ast,+)$ or $(\ast,+)^r$. 
The latter holds because for $a$ to exhibit types with distinct second coordinates (other than $\pm$) under, say, $f^q$ and $f^n$, one of the alternatives~(b) or~(d) in item~\ref{item: types of fixed points 6} of Lemma~\ref{lem: local type of fixed points} must apply.
Since neither $f^{2q}$ nor $g^{2q}$ crosses the diagonal in $(a,b)$ (Fact~\ref{fact: no crossings in Un-component} applies because $2q \mid n$), and since $f^{2q}(x_0), g^{2q}(x_0) \neq x_0$ (as $P > n \geq 2q$), we conclude that $x_0 \in U_{2q}$.
Minimality of $n$ then forces $n = 2q$, and since $q' > q$ divides $n$, we deduce that $q' = n$.
The statement follows.
\end{proof}

Due to Claim~\ref{claim: boundary points of Un components are n-periodic}, we may assume without loss of generality that $a\in \Per_n(f,g)$ in the following; the case with $b\in \Per_n(f,g)$ is similar.

Note that for $n<p<P$ with $n\nmid p$ and $x_0$ as above (with corresponding $U_n$-component $(a,b)$), minimality of $P$ implies that $\Per_q(f,g)\cap (a,b)=\emptyset$ for any divisor $q$ of $p$---this follows from the first part of Fact~\ref{fact: Un-components dont contain endpoints} if $q\nmid n$ and is an immediate consequence of the definition of $U_n$-components if $q\mid n$.
Since $\PER^s(f)\ssq \PER(g)$ and $\PER^s(g)\ssq \PER(f)$, this implies that neither $f^p$ nor $g^p$ crosses the diagonal in $(a,b)$.
Note further that as $p\bmod n<n$
and hence $U_{p\bmod n}\cap \Per_n(f,g)=\emptyset$ (by minimality of $n$), we either have
$f^p(a)=f^{p\bmod n}(a)> a$ and 
$g^p(a)=g^{p\bmod n}(a)>a$ or 
$f^p(a)< a$ and 
$g^p(a)< a$; note that equality is ruled out due to $a\in \Per_n(f,g)$ (and thus $f^p(a),g^p(a)\neq a$).
Combining this observation with the fact that neither $f^p$ nor $g^p$ crosses the diagonal in $(a,b)$, we obtain $U_p \cap (a,b) = \emptyset$. 
In particular, this implies
\begin{align}\label{eq: x0 not in Up for p not multiple of n}
    x_0\notin U_p \quad \text{whenever } n<p<P \text{ and } n\nmid p.
\end{align}

We conclude from the above that only multiples of $n$ can contribute to the difference in \eqref{eq: contradiction to same number of left visits}.
Now, by Fact~\ref{fact: no crossings in Un-component} and Fact~\ref{fact: same type in Un-component} we may assume without loss of generality that $f^n(x)\geq x$ (and hence $g^n(x)\leq x$) for all $x\in (a,b)$, that is, $a$ is of type $(\ast,+)$ under $f^n$ (the case $(\ast,+)^r$ is analogous; if either condition held for 
$g^n$, we would replace $f^n$ by $g^n$ in what follows).
Then, also $f^{\ell n}(x)\geq x$ ($x\in(a,b)$) as long as $\ell n<P$ since on the one hand, $a$ remains
of type $(\ast,+)$ for all iterates of $f^n$---see Lemma~\ref{lem: local type of fixed points} \ref{item: types of fixed points 3}---and on the other hand, $f^{\ell n}$ does not cross the diagonal on $(a,b)$ as $P$ is minimal. 
Indeed, a crossing would imply the existence of a stable fixed point $y$ of $f^{\ell n}$ (and hence of $g^{\ell n}$, since $\PER^s(f)\ssq \PER(g)$) in $(a,b)$.
By the definition of $U_n$-components, $y$ could not be $p$-periodic for any divisor $p$ of $n$. 
Therefore, by Fact~\ref{fact: Un-components dont contain endpoints}, we would have $\Per_p(f,g) \cap U_n \neq \emptyset$ for some $p < P$ with $p \nmid n$, contradicting the minimality of $P$.
This shows that if $x_0$ is mapped to its right under $f^n$, then it is also mapped to its right under $f^{\ell n}$ for $\ell n<P$.

Altogether, assuming without loss of generality that $f^n(x_0)>x_0$ and $g^n(x_0)<x_0$, the above implies
\begin{align*}
 &\#\{1\leq m<P\:f^m(x_0)>x_0\}-\#\{1\leq m<P\:g^m(x_0)>x_0\}\\
 &\stackrel{\eqref{eq: x0 not in Up for p not multiple of n}}{=}
 \#\{\ell\in \N\:\ell n<P \text{ and } f^{\ell n}(x_0)>x_0\}-\#\{\ell \in \N\:\ell n<P \text{ and } g^{\ell n}(x_0)>x_0\}
 \geq 1.
\end{align*}
That is, \eqref{eq: contradiction to same number of left visits} holds and the statement follows.
\end{proof}

\section{Applications}\label{sec: applications}
In the first subsection, we apply Lemma~\ref{lem: left-right determines map} and Theorem~\ref{thm: periodic orbits determine left-right behaviour} to settings where the equality $\PER(f)=\PER(g)$ yields particularly strong consequences, specifically for endomorphisms with dense periodic points. 
In the subsequent section, we then use these applications to derive implications of this equality for general endomorphisms on compact $\I$.
\subsection{Mixing endomorphisms and endomorphisms with dense periodic points}\label{sec: applications dense orbits}
The next statement---a strengthening of Theorem~\ref{introthm: mixing maps are determined by period spectrum} from the introduction---is not only the most direct consequence of the results in the previous sections, but also underpins all subsequent implications in this article.

Recall that given two endomorphisms $f$ and $g$ on $\I$, we write $\Per_p(f,g)$ for the collection of points whose orbits are in $\PER_p(f)\cap \PER_p(g)$.
\begin{cor}\label{cor: mixing maps are fully determined by periodic orbits}
    Suppose $f, g \colon \I \to \I$ are continuous with
$\PER^s(f) \subseteq \PER(g)$ and $\PER^s(g) \subseteq \PER(f)$.
If $\overline{\bigcup_{p \text{ prime}} \Per_p(f,g)} = \I$, then $f = g$.

In particular,
if $f, g \colon \I \to \I$ are continuous with
$\PER^s(f) \subseteq \PER(g)$ and $\PER^s(g) \subseteq \PER(f)$, and $f$ is topologically mixing, then $f = g$.
\end{cor}
\begin{rem}\label{rem: mixing maps are fully determined by periodic points}
For use in the proof of Corollary~\ref{cor: dense periodic orbits determine second iterate}, let us note that the first part of the above statement remains valid if, instead of assuming denseness of $\bigcup_{p \text{ prime}} \Per_p(f,g)$, we assume 
\[\overline{\Per_1(f,g)\cup \bigcup_{p \text{ prime}} \Per_p(f,g)}=\I.\]
Although points in $\Per_1(f,g)$ are not dealt with explicitly in the proof below, this slight extension follows immediately from it.
\end{rem}
\begin{proof}
We start with the first part.
By assumption, the collection of points with orbits in 
$\PER(f)\cap\PER(g)$ is dense, and thus, by 
Theorem~\ref{thm: periodic orbits determine left-right behaviour}, 
$U_n=\emptyset$ for all $n\in\N$. 
In particular, for any prime periodic orbit $\mc O(x)$ contained in 
$\PER(f)\cap\PER(g)$, the left-right behaviour of $\mc O(x)$ under 
$f$ coincides with that under $g$.
From Lemma~\ref{lem: left-right determines map}, we obtain that
$\left.f\right|_{\mc O(x)}=\left.g\right|_{\mc O(x)}$.
As $\bigcup_{p \text{ prime}}\Per_p(f,g)$ is dense in $\I$, this shows that $f$ and $g$ coincide on a dense set of points and hence, by continuity, $f=g$.

For the second part, we prove that $\bigcup_{p \text{ prime}} \Per_p(f,g)$ is dense in $\I$.
To this end, it suffices to show that $\bigcup_{p \text{ prime}} \Per_p^s(f)$ is dense in $\I$, since $\PER^s(f) \subseteq \PER(g)$.

Now, given any interval $J\ssq\I$, pick a closed subinterval $K\ssq J$ 
containing no fixed points of $f$; this is possible since $\Fix_1(f)$ 
is nowhere dense under the assumption of topological mixing---it is 
clearly closed and cannot contain a proper interval.
Suppose $K=[a,b]$ and assume without loss of generality that there are $\0<a'<a$ and $b<b'<\1$ with $K\cap \big((\0,a')\cup (b',\1)\big)=\emptyset$.

As $f$ is mixing, there exist $N_1,N_2\in\N$ such that 
$f^n(K)\cap (b',\1)\neq\emptyset$ for all $n\geq N_1$ and 
$f^n(K)\cap (\0,a')\neq\emptyset$ for all $n\geq N_2$. 
It follows that $K\ssq [a',b']\ssq f^n(K)$ for all $n\geq N\coloneqq \max(N_1,N_2)$; in particular,
$f^n$ crosses the diagonal on $K$.
Consequently, $K\cap\Fix^s_n(f)\neq\emptyset$ for all $n\geq N$. 
In particular, for any prime $p\geq N$, we have 
$K\cap\Fix_p^s(f)\neq\emptyset$ and, since $K\cap\Fix_1(f)=\emptyset$, 
it follows that $K\cap\Per^s_p(f)\neq\emptyset$. 
As $J\ssq\I$ was an arbitrary interval and $K\ssq J$, this completes the proof.
\end{proof}

\begin{rem}
Let us point out that Corollary~\ref{cor: mixing maps are fully determined by periodic orbits} provides an answer to a question regarding the \emph{bifurcation set} originally posed by the present authors in~\cite{FuhrmannGrogerPasseggi2021}. 
Specifically, it was asked whether there exist natural families of maps on the unit interval $[0,1]$ for which the bifurcation set is a \emph{complete topological invariant}---that is, whether the existence of a homeomorphism between bifurcation sets implies topological conjugacy of the maps. 
While we refrain from a formal introduction of the bifurcation set here (referring the reader to~\cite{FuhrmannGrogerPasseggi2021} for details), 
it should be noted that the bifurcation set of any transitive interval map uniquely identifies its set of periodic orbits. 
In light of Corollary~\ref{cor: mixing maps are fully determined by periodic orbits}, it follows that the bifurcation set is indeed a complete topological invariant for the class of topologically mixing interval maps. 
Indeed, the results in the present work were also motivated by an interest in understanding the extent to which the bifurcation set determines the underlying map, or equivalently, the degree to which distinct maps are characterised by distinct bifurcation sets.
\end{rem}

We now turn to a strengthening of
Theorem~\ref{introthm: dense periodic orbits determine second iterate} from the introduction.
In the following, we write $\Per(f,g)$ for the collection $\bigcup_{p \in \N} \Per_p(f,g)$ of all periodic points with orbits in $\PER(f)\cap\PER(g)$; similarly, we write $\Per(f)$ for the collection $\bigcup_{p \in \N} \Per_p(f)$ of all periodic points of $f$.
\begin{cor}\label{cor: dense periodic orbits determine second iterate}
Suppose $f,g\: \I\to\I$  are continuous with $\PER^s(f)\ssq \PER(g)$ and $\PER^s(g)\ssq \PER(f)$.
If $\overline{\Per(f,g)}=\I$, then $f^2=g^2$.
\end{cor}

To prove Corollary~\ref{cor: dense periodic orbits determine second iterate}, we require two classical results originally formulated for compact $\I$.
The first is due to Barge and Martin \cite{BargeMartin1985}; see also the monograph by Ruette \cite{Ruette2017} for a comprehensive discussion.
Our formulation below is closer to that of Coven and Mulvey \cite{CovenMulvey1988}, whose alternative proof extends to arbitrary intervals with only small modifications (see the appendix for details).

In the following, we say that $f\colon \I \to \I$ \emph{separates} $\I$ if there exists a decomposition $\I = L \cup C \cup R$ into closed intervals (closed as subsets of $\I$) with disjoint interiors, where $C$ may be degenerate (while $L$ and $R$ are non-degenerate), such that $f(L) \subseteq R$, $f(C) = C$, and $f(R) \subseteq L$.
We may call a decomposition with these properties also a \emph{separation} corresponding to $f$.
When given such a separation, we may assume without loss of generality that $L$ lies to the left of $C$ and $C$ lies to the left of $R$.

\begin{thm}\label{thm: CovenMulvey}
Let $f \colon \I \to \I$ be continuous, and suppose that $\overline{\Per(f)} = \I$.
If $f$ does not separate $\I$, then there exists a collection (possibly finite or empty) $J_1,J_2,\ldots$ of non-degenerate, closed (as subsets of $\I$), invariant subintervals of $\I$, with pairwise disjoint interiors, such that
\begin{itemize}
    \item $\left.f\right|_{J_i}$ is topologically mixing for each $i$, and
    \item $\left.f\right|_{\I \setminus \bigcup_i J_i}$ is the identity map.
\end{itemize}
If $f$ separates $\I$, $f^2$ does not, and the conclusion of the preceding paragraph holds for $f^2$.
\end{thm}
Suppose $f$ satisfies the above assumptions but separates $\I$ (with $L$ to the left of $C$ and $C$ to the left of $R$), and let $J_1, J_2, \ldots$ be the intervals obtained by applying the previous theorem to $f^2$.
It is easy to see that for each $J_i \subseteq L$, there exists $J_{i'} = f(J_i)$ in $R$, and if $J_i \subseteq L$ lies to the right of some $J_k$, then $J_{i'}$ lies to the left of the corresponding $J_{k'}$.
The analogous statement holds for the intervals $J_i$ in $R$.
\begin{fact}\label{fact: smallest separation}
In the situation and with the notation of the previous observation, let $K$ denote the closure of the union of all fixed points of $f$ and all intervals $J_i$ satisfying $f(J_i) = J_i$.
It is not hard to see that $K$ is an interval (possibly degenerate) that gives rise to a separation $\I = L' \cup K \cup R'$ corresponding to $f$.
Note that $f$ does not separate $K$.
\end{fact}
The second result we need was originally announced by Blokh in an unpublished paper \cite{Blokh1982} (see \cite{Blokh1995} for a published version) and was later independently discovered by Barge and Martin \cite{BargeMartin1985b, BargeMartin1987}.
The formulation below, which applies to arbitrary intervals, follows from \cite[Theorem~6.1]{Banks1997} and \cite[Theorem~7.1]{Banks1997} (see also the paragraph preceding the latter in \cite{Banks1997}).
Note that it is also used in the proof of Theorem~\ref{thm: CovenMulvey} discussed in the appendix.

Recall that an endomorphism $f$ on $\I$ is \emph{(topologically) transitive} if
for each pair of open sets $U,V\ssq \I$ there is $n\in \N$ such that $f^nU\cap V\neq \emptyset$.
\begin{thm}\label{thm: transitive vs mixing maps}
Let $f$ be a transitive endomorphism on $\I$. 
Then one of the following holds.
\begin{itemize}
    \item The map $f$ is topologically mixing.
    \item There exists $c \in (\0,\1)$ such that $c$ is the unique fixed point of $f$, $f(\llbracket \0,c]) = [c,\1\rrbracket$, $f([c,\1\rrbracket) = \llbracket \0,c]$, and both maps $f^2|_{\llbracket \0,c]}, f^2|_{[c,\1\rrbracket}$ are topologically mixing. 
\end{itemize}
\end{thm}
\begin{rem}\label{rem: equivalent characterisations of transitivity}
As a consequence of the previous statement, one obtains the well-known result that transitive maps have dense periodic points \cite{Sharkovsky1964}; see also \cite{Ruette2017} for a direct proof in the case of compact $\I$.
It is moreover a classical fact that an endomorphism $f$ on $\I$ is transitive if and only if there exists $x \in \I$ such that $\mc O(x)$ is dense in $\I$.
Combining these two observations, we conclude that an endomorphism $f$ on $\I$ is transitive if and only if, for every compact interval $I \subseteq \I$ and every $\eps > 0$, there exists $x \in \Per(f)$ whose orbit intersects $I$ in an $\eps$-dense set.
In fact, since $f$ or $f^2$ (when restricted to $\llbracket \0,c]$ with $c$ as in Theorem~\ref{thm: transitive vs mixing maps}) is mixing, we find stable periodic points arbitrarily close to such $x$ (similar to how we showed 
denseness of stable prime periodic points in the proof of Corollary~\ref{cor: mixing maps are fully determined by periodic orbits}). 
It follows that $f$ is transitive if and only if for every compact interval $I \subseteq \I$ and every $\eps > 0$, there exists $x \in \Per^s(f)$ whose orbit intersects $I$ in an $\eps$-dense set.
 \end{rem}
\begin{fact}\label{fact: transitivity}
    Note that if $f$ separates $\I$ with a decomposition in which $C$ is degenerate and $\I$ is the union of exactly two distinct intervals $J_1$ and $J_2$ on which $f^2$ is mixing, then $f$ is transitive (and we are in the second case of Theorem~\ref{thm: transitive vs mixing maps}).
    Indeed, $f^2$ admits a dense orbit in $J_1$ (mixing implies transitivity), and hence, as $f(J_1) = J_2$, $f$ admits a dense orbit in $\I$.
\end{fact}
\begin{proof}[Proof of Corollary~\ref{cor: dense periodic orbits determine second iterate}]
We first assume that $f$ does not separate $\I$ and use the notation from Theorem~\ref{thm: CovenMulvey}, where the intervals $J_i$ correspond, a priori, to $f$.
Note that the proof of the second part of Corollary~\ref{cor: mixing maps are fully determined by periodic orbits} shows that, for each $J_i$, we have $\overline{\bigcup_{p \text{ prime}} \Per_p^s(\left.f\right|_{J_i})} = J_i$.
Therefore, since $\PER^s(\left.f\right|_{J_i}) \ssq \PER^s(f) \ssq \PER(g)$, it follows that
\[
\overline{\bigcup_{p \text{ prime}} \Per_p(f,g)} = \overline{\bigcup_i \bigcup_{p \text{ prime}} \Per_p(\left.f\right|_{J_i},\left.g\right|_{J_i})} = \overline{\bigcup_i J_i}.
\]
Moreover, since $f$ is the identity on $\I \setminus \bigcup_i J_i$, and $\overline{\Per(f,g)}=\I$, we necessarily have
\[
\overline{\Per_1(f,g)} \cap \big(\I \setminus \overline{\bigcup_i J_i}\big) = \overline{\Per(f,g)} \cap \big(\I \setminus \overline{\bigcup_i J_i} \big)= \I \setminus \overline{\bigcup_i J_i}.
\]
We conclude that $\Per_1(f,g) \cup \bigcup_{p \text{ prime}} \Per_p(f,g)$ is dense in $\I$.
Now $f = g$ (and hence, trivially, $f^2 = g^2$) follows with Remark~\ref{rem: mixing maps are fully determined by periodic points}.

Suppose now that $f$ separates $\I$ with separation $\I = L \cup C \cup R$, and consider the sequence $J_1, J_2, \ldots$ of intervals corresponding to $f^2$ (as given by Theorem~\ref{thm: CovenMulvey}).
First, we observe that similar as in the previous case, $f = g$ on $\I \setminus \bigcup J_i$.
Indeed, this equality holds trivially on $\Per_1(f,g)\cup \Per_2(f,g)$ and extends to all of $\I \setminus \bigcup J_i$ by continuity, since by assumption,
\[
\left( \overline{\Per_1(f,g)\cup \Per_2(f,g)}\right) \cap \left(\I \setminus \bigcup J_i\right) = \overline{\Per(f,g)} \cap \left(\I \setminus \bigcup J_i\right) = \I \setminus \bigcup J_i.
\]
Here, we used that points in $\I \setminus \bigcup J_i$ are fixed under $f^2$ and thus belong to
$\Per_1(f)\cup \Per_2(f)$.

To proceed with the points in $\bigcup J_i$, note that each $J_i$ is contained in either $L$, $C$, or $R$. 
This follows since $J_i \cap L$, $J_i \cap C$, and $J_i \cap R$ are invariant under $f^2$, meaning that two of these intersections must have empty interior for $f^2$ to be mixing on $J_i$.

Suppose first that $J_i\ssq C$.
Without loss of generality, we may assume that $f$ does not separate $C$ (Fact~\ref{fact: smallest separation}).
Now, $\overline{\Per(f,g)}=\I$ implies that $C$ is invariant not only under $f$ but also under $g$ (see also the next paragraph for similar reasoning), so that the equality of $f^2$ and $g^2$ on $J_i$ follows as in the first part of the proof.

It remains to consider the case where $J_i \subseteq L$ (the case of $J_{i} \subseteq R$ is analogous).
Observe that $J_i\cup f(J_i)$ is invariant under $g$: every $f$-periodic orbit meeting $J_i\cup f(J_i)$ is entirely contained in it, so by denseness of $\Per(f,g)$ we obtain a dense set of $g$-periodic orbits contained in $J_i\cup f(J_i)$, giving the desired invariance.

In fact, we obtain that $g(J_i) = f(J_i)$ (and hence $g(f(J_i))=J_i$ and $g^2(J_i)=J_i$).
If $C$ is non-degenerate, this follows since 
\[g\big(\Per(f,g) \cap J_i\big) \cap \big(\Per(f,g) \cap f(J_i)\big)\neq \emptyset\] 
while $C\cap g(J_i)$ must have empty interior (so that, in fact, $g\big(\Per(f,g) \cap J_i\big) = \Per(f,g) \cap f(J_i)$); a similar argument applies if $C$ is degenerate and its unique element is not an endpoint of $J_i$.
If $C$ is degenerate and not disjoint from $J_i$ (and hence not disjoint from $f(J_i)$), then $g(J_i) = f(J_i)$ is a consequence of the transitivity of $\left.f\right|_{J_i \cup f(J_i)}$ (Fact~\ref{fact: transitivity}) and Remark~\ref{rem: equivalent characterisations of transitivity}: together, they give the existence of arbitrarily dense stable periodic orbits of $f$ and hence of arbitrarily dense periodic orbits of $g$ (since $\PER^s(f) \subseteq \PER(g)$), which in turn implies the transitivity of $\left.g\right|_{J_i \cup f(J_i)}$ (again Remark~\ref{rem: equivalent characterisations of transitivity}).
We note that $\left.g\right|_{J_i\cup f(J_i)}$ cannot be mixing, for if it were, then Corollary~\ref{cor: mixing maps are fully determined by periodic orbits} would imply that also $\left.f\right|_{J_i\cup f(J_i)}$ is mixing.
Observing that the unique fixed point of $f$ is necessarily stable (Theorem~\ref{thm: transitive vs mixing maps}) and hence also the unique fixed point of $g$, we obtain the desired equality (see again Theorem~\ref{thm: transitive vs mixing maps}).

Now, note that for any point $x\in \Per(\left.f^2\right|_{J_i})$, $\mc O_{f^2}(x)$ is given by $\mc O_f(x)\cap J_i$.
As a consequence, 
$\PER^s(\left.f^2\right|_{J_i})\ssq \PER(\left.g^2\right|_{J_i})$.
Indeed, given $\mc O_{f^2}(x)\in \PER_p^s(\left.f^2\right|_{J_i})$, there is some $y\in \mc O_{f^2}(x)\cap \Fix^s_p(f^2)=\mc O_{f^2}(x)\cap \Fix^s_{2p}(f)$ and hence  $y\in \mc O_{f}(x)\cap \Fix^s_{2p}(f)$; in other words, 
$\mc O_{f}(x)\in \PER^s(f)\ssq \PER(g)$, that is, $\mc O_f(x)=\mc O_g(x)$.
This shows the desired inclusion since also
$\mc O_{g^2}(x)=\mc O_g(x)\cap J_i$ (for by the above, $g(J_i)=f(J_i)$ and $g(f(J_i))=J_i$).
We similarly obtain $\PER^s(\left.g^2\right|_{J_i})\ssq \PER(\left.f^2\right|_{J_i})$.
With the second part of Corollary~\ref{cor: mixing maps are fully determined by periodic orbits}, this implies $\left.f^2\right|_{J_i}=\left.g^2\right|_{J_i}$ for each of the above intervals $J_i$.
\end{proof}

\subsection{General endomorphisms on compact intervals}
We now take a look at the implications of Lemma~\ref{lem: left-right determines map} and Theorem~\ref{thm: periodic orbits determine left-right behaviour} for general endomorphisms on compact intervals $\I$ to show that, locally, wherever there is entropy, the period spectrum determines a power of the respective map.
Concretely, we will prove a strengthened version of Theorem~\ref{introthm: basic sets}.
To that end, we begin by revisiting some terminology and background concerning the ergodic decomposition of endomorphisms on compact intervals due to Blokh \cite{Blokh1995}.

\subsubsection{Background: spectral decomposition of interval maps}
Recall that the \emph{$\omega$-limit (set)} of a point $x$ under a map $f$ is the invariant set given by
\[\omega_f(x)=\omega(x) = \bigcap_{n=0}^{\infty} \overline{\{f^k(x) : k \ge n\}}.\]
On compact intervals $\I$, $\omega$-limits behave particularly well: given an endomorphism $f$ on $\I$, the collection of all $\omega$-limits $\omega_f=\{\omega_f(x)\:x\in \I\}$ is compact when equipped with the Hausdorff metric \cite{BlokhBrucknerHumkeSmital1996}.
As a direct consequence, every chain $C$ of $\omega$-limit sets of $f$---every subset $C$ of $\omega_f$ such that for any pair $\omega, \omega' \in C$,  $\omega \subseteq \omega'$ or $\omega' \subseteq \omega$ holds---possesses an upper bound $\overline{\omega} \in \omega_f$, that is, $\omega \subseteq \overline{\omega}$ for every $\omega \in C$, see also \cite{Sharkovsky1966}.
It follows from Zorn's lemma that every $\omega$-limit set is contained in a maximal (with respect to set inclusion) $\omega$-limit set.

Blokh's spectral decomposition theorem implies that there are essentially three types of maximal $\omega$-limit sets. 
We refer the reader to \cite{Blokh1995} for full details, recapping here only what is strictly necessary for what follows. 

First, there are periodic orbits that form maximal $\omega$-limit sets; we denote the union of all such orbits by $X_f$. 
Second, there are \emph{solenoidal sets}, which extend odometers and thus lack periodic points; whilst a solenoidal set $S$ is not necessarily an $\omega$-limit set in its own right, it contains a unique maximal $\omega$-limit set, which is itself solenoidal and, moreover, coincides with $\omega(x)$ for each point $x\in S$. 
Finally, \emph{basic sets} are the only maximal $\omega$-limit sets with positive topological entropy (see \cite{AdlerKonheimMcAndrew1965} for the definition of topological entropy).

\begin{thm}[cf. {\cite[Theorem~5.4]{Blokh1995}}]\label{thm: spectral decomposition}
Let $f : \I \to \I$ be a continuous map. 
Then there exist an at most
countable family of basic sets $\{B_i\}$
and a family of solenoidal sets
$\{S^{(\alpha)}\}$ such that
\[
\overline{\Per(f)} = X_f \cup
          \bigcup_{\alpha} S^{(\alpha)}
          \cup \bigcup_{i} B_i,
\]
where, on the right-hand side, intersections are possible only between
different basic sets; any three such sets have an empty intersection and the intersection of two basic sets is finite.
\end{thm}
\begin{rem}\label{rem: only basic sets have several periodic orbits}
    It follows from the above that if distinct periodic orbits belong to one and the same maximal $\omega$-limit set $\omega$, then $\omega$ is necessarily a basic set.
    Indeed, the omega limit sets contained in $X_f$ are single periodic orbits and those contained in the union of solenoidal sets are themselves solenoidal and, consequently,
    contain no periodic orbits.
\end{rem}

We collect some properties of basic sets and refer the reader to \cite{Blokh1995} for additional background.
First of all, each basic set $B$ is perfect (no isolated points).
Second, each basic set $B$ corresponds to a \emph{periodic interval} $J$, that is,
a closed interval $J$ for which there is $p\in \N$ with $f^pJ=J$ and $f^iJ\cap f^kJ=\emptyset$ for $i\neq k\in \{0,\ldots,p-1\}$.
Specifically, writing $J^*=\bigsqcup_{k=0}^{p-1} f^kJ$, the relation between $B = B_f(J^*)$ and $J$ is given by
\[
B_f(J^*) = \big\{ x \in J^* \:
\text{for every relative neighbourhood } U \text{ of } x \text{ in } J^* \text{ we have }
J^* = \overline{\bigcup_{k\geq 0} f^k(U)} \big\}.
\]
Further, for each $k\in \{0,\dots,p-1\}$, there is a monotone semiconjugacy $\phi_k\: f^kJ\to \J$
between $\left.f^{p}\right|_{f^kJ}$ and a transitive endomorphism $F_k$ on some compact interval $\J$, that is,
$\phi_k$ is monotone, continuous, surjective, and satisfies $\phi_k\circ\left.f^{p}\right|_{f^kJ}=F_k\circ \phi_k$ \cite[Theorem~4.1]{Blokh1995}.
This semiconjugacy is such that 
\begin{align}\label{eq: only endpoints of gaps are identified}
I\ssq f^kJ \text{  is an interval with } \phi_k I \text{ a singleton if and only if } \mathring I\cap B=\emptyset, 
\end{align}
see \cite[Step B8 (p.~40)]{Blokh1995}.
We may call an interval $I$ as in \eqref{eq: only endpoints of gaps are identified} a \emph{gap} of $B$ if it is maximal (that is, not strictly contained in another such interval) and refer to $F_k$ as an \emph{interval model} of $\left.f^{p}\right|_{f^kJ}$.
One readily obtains the following
\begin{fact}\label{fact: periodic orbits of basic sets}
In the above situation and using the above notation, suppose $x$ is a $q$-periodic point of $F_k$ for some $k \in \{0, \dots, p-1\}$. 
If $\#\phi_k^{-1}(F_k^i x) > 1$ for every $i \in \{0, \dots, q-1\}$ (which means that $\phi_k^{-1}(F_k^i x)$ coincides with a gap for every $i$), then either $\phi_k^{-1}(x)$ intersects a unique $2q$-periodic orbit of $\left.f^p\right|_{f^kJ\cap B}$, or it intersects at least one and at most two $q$-periodic orbits of $\left.f^p\right|_{f^kJ\cap B}$.
If, however, $\#\phi_k^{-1}(F_k^i x) = 1$ for some $i \in \{0, \dots, q-1\}$, then
$\phi_k^{-1}(x)$ intersects a sole $q$-periodic orbit of $\left.f^p\right|_{f^kJ\cap B}$.
\end{fact}
We will also require the following observation; see \cite[Step B8 (p.~40)]{Blokh1995}.
\begin{fact}\label{fact: all collapsing maps are factor maps}
    Any continuous monotone surjection $\phi_k\: f^kJ\to \J$ (with $\J$ some compact interval) which satisfies \eqref{eq: only endpoints of gaps are identified} gives rise to a continuous transitive endomorphism $F_k=\phi_k\circ \left.f^p\right|_{f^kJ}\circ \phi_k^{-1}$ on $\J$.
\end{fact}

Finally, for each basic set $B$, $\left.f\right|_{B}$ is transitive \cite[Theorem~4.1]{Blokh1995} and $\overline{\Per(f)\cap B}=B$ \cite[Theorem~10.3]{Blokh1995}, which together imply
\begin{prop}\label{prop: dense periodic orbits in B}
    Suppose $B$ is a basic set of $f$.
    For each $n\in\N$, there is a periodic point $x_n$ whose orbit is $1/n$-dense in $B$.
\end{prop}
\subsubsection{Identifying basic sets and their dynamics}
The properties of basic sets mentioned above readily imply the first part of 
Theorem~\ref{introthm: basic sets}: if $f$ and $g$ are two endomorphisms on $\I$ 
with $\PER(f)=\PER(g)$, then the basic sets of $f$ and $g$ coincide.
Indeed, given a basic set $B$ of $f$ and any two periodic points $y,z \in B$ with $\mc O_f(y)\neq \mc O_f(z)$, Proposition~\ref{prop: dense periodic orbits in B} ensures that there exists a sequence of $f$-periodic points $(x_n)\ssq B$ whose orbits accumulate on both $y$ and $z$. 
Since $\mc O_f(x_n)=\mc O_g(x_n)$ and because periodic orbits are themselves $\omega$-limit sets, 
the compactness of $\omega_g=\{\omega_g(x)\: x\in \I\}$ implies that $y$ and $z$ 
lie in a common $\omega$-limit set of $g$, and consequently belong to the same basic 
set of $g$ (see Remark~\ref{rem: only basic sets have several periodic orbits}). 
Since $y$ and $z$ were arbitrary and since $\Per(f)\cap B$ is dense in $B$, it follows that $B$ is contained in a basic set of $g$. 
By symmetry, we conclude that any basic set of $f$ is a basic set of $g$, and 
vice versa.

With the above, Theorem~\ref{introthm: basic sets} becomes an immediate consequence of the next statement, where we write $\PER_B(f)$ for the collection of $f$-periodic orbits 
contained in $B$.
\begin{thm}\label{thm: main theorem on basic sets}
Suppose $f$ and $g$ are endomorphisms on $\I$ 
which have a basic set $B$ in common.
Specifically, let $J$ be a $p$-periodic interval of $f$ and $J'$ be a $p'$-periodic interval of $g$ such that 
\[
B=B_f\Big(\bigsqcup_{k=0}^{p-1}f^k J\Big)=B_g\Big(\bigsqcup_{k=0}^{p'-1}g^k J'\Big).
\]
Suppose that $\PER_B(f)=\PER_B(g)$. 
Then for $q=2\min\{p,p'\}$, we have that
$\left.f^q\right|_B=\left.g^q\right|_B$.
\end{thm}
\begin{rem}
Note that $B$ is either a disjoint union of intervals or a Cantor set \cite[Corollary~4.4]{Blokh1995}.
If $B$ is a disjoint union of intervals, the proof of the above statement is straightforward. 
Indeed, in this setting, one has $B\cap f^kJ=f^kJ$ and $B\cap g^kJ'=g^kJ'$ for 
each $k$, which forces $B=\bigsqcup_{k=0}^{p-1}f^k J=\bigsqcup_{k=0}^{p'-1}g^k J'$ and, consequently, $p=p'$. 
Since $\overline{\Per(f)\cap B}=B$
ensures that $\overline{\Per(\left.f^p\right|_{f^kJ})\cap f^kJ}=f^kJ$ for each $k$,
Corollary~\ref{cor: dense periodic orbits determine second iterate} implies that
$f^{2p}$ and $g^{2p}$ coincide on $B$.

However, if $B$ is a Cantor set, the proof is more involved.
The first complication in this setting is that we can no longer guarantee that 
$\bigsqcup_{k=0}^{p-1}f^k J = \bigsqcup_{k=0}^{p'-1}g^k J'$, nor can we ensure that $p=p'$.
Furthermore, to apply the results of Section~\ref{sec: applications dense orbits} to 
establish that the restrictions of $f^q$ and $g^q$ to $B$ coincide (where $q=2\min\{p,p'\}$), we must utilise underlying interval models $F_k$ and $G_k$, respectively. 
The case $p > p'$ is somewhat more involved; here, we show that, in fact, $p = 2p'$ and that the family of interval models $(F_k)_{k=0}^{2p'}$ is obtained by appropriately restricting the second iterate of each interval model in $(G_k)_{k=0}^{p'}$ to two subintervals.
Observe that we do not explicitly distinguish between the two cases in the proof below.
\end{rem}
\begin{proof}
We begin by showing that either $p=p'$, $p=2p'$, or $2p=p'$. 
To this end, let $Q$ denote the greatest common divisor of the periods of 
all periodic orbits of $\left.f\right|_B$ (and hence of $\left.g\right|_B$). 
Since every $P$-periodic orbit ($P\in \N$) of $\left.f\right|_B$ intersects 
each interval $f^k J$ for $k \in \{0, \dots, p-1\}$ in exactly $P/p$ points, 
it follows that $p$ must divide $Q$. 
In fact, $Q\in\{p,2p\}$; indeed, assuming $Q=np$ for some $n>2$ would imply 
that all periodic orbits of $\left.f^p\right|_{B\cap f^kJ}$ have period at least $3$. 
By Fact~\ref{fact: periodic orbits of basic sets}, this would mean that any interval model $F_k$ lacks a fixed point---a contradiction, since every endomorphism on a compact interval $\J$ admits a fixed point.
By a symmetric argument, $Q\in\{p',2p'\}$, which proves the claim.

We will consider the cases $p=p'$ and $p=2p'$ (and, by symmetry, $2p=p'$) 
separately. 
To address the fact that in general---and even when $p=p'$---the 
interval orbits $\bigsqcup_{k=0}^{p-1}f^k J$ and $\bigsqcup_{k=0}^{p'-1}g^k J'$ 
do not coincide, we first introduce some terminology. 
We say that a collection of points $a_1,\dots,a_{p-1}\in \I\setminus B$ are 
\emph{$\left.f\right|_B$-separators} if, setting $a_0=\0$ and $a_p=\1$ (recall that $\I=[\0,\1]$), we have 
$a_0<a_1<\dots<a_{p-1}<a_p$ and there exists a permutation $\sigma$ on 
$\{0,\dots,p-1\}$ such that 
\[f^{\sigma(k)}J\cap B=[a_k,a_{k+1}]\cap B \qquad (k\in\{0,\dots,p-1\}).\] 
We define \emph{$\left.g\right|_B$-separators} analogously, with the 
corresponding permutation on $\{0,\dots,p'-1\}$ denoted by $\tau$. 
The existence of such separators is straightforward.

First, consider the case where $p=p'$. 
We begin by showing that, in this setting, any collection of 
$\left.f\right|_B$-separators is also a collection of 
$\left.g\right|_B$-separators, and conversely.
Indeed, suppose for a contradiction that $a_1,\dots,a_{p-1}$ satisfied $f^{\sigma(k)}J\cap B=[a_k,a_{k+1}]\cap B$ but
$g^{\tau(k)}J'\cap B\neq[a_k,a_{k+1}]\cap B$ and let $k$ be minimal with this property.
Without loss of generality, we may assume that $g^{\tau(k)}J'\cap B\subsetneq[a_k,a_{k+1}]\cap B$.
By denseness of $g$-periodic points in $B$ (see Proposition~\ref{prop: dense periodic orbits in B}) and hence, denseness of $g$-periodic points in $[a_k,a_{k+1}]\cap B\setminus g^{\tau(k)}J'$, this would imply the existence of a $g$-periodic orbit whose intersection with $g^{\tau(k)}J'$ was of smaller cardinality than its intersection with $f^{\sigma(k)}J$---a contradiction.
This shows equality of $\left.f\right|_B$- and 
$\left.g\right|_B$-separators and, consequently,
$f^{\sigma(k)}J\cap B=g^{\tau(k)}J'\cap B$ for $k\in \{0,\dots,p-1\}$.

To finish the proof for $p=p'$, we now show that $\left.f^{2p}\right|_{f^kJ\cap B}=\left.g^{2p}\right|_{f^kJ\cap B}$ for all $k\in\{0,\dots,p-1\}$. 
To that end, let $\phi_k\: f^kJ\to \I$ be a continuous monotone surjection satisfying \eqref{eq: only endpoints of gaps are identified}.
Then $\phi_k$ acts as a semiconjugacy between $\left.f^p\right|_{f^kJ\cap B}$ and a transitive endomorphism $F_k$ on $\I$ and likewise, as a semiconjugacy between $\left.g^p\right|_{f^kJ\cap B}$ and a transitive endomorphism $G_k$ on $\I$, see Fact~\ref{fact: all collapsing maps are factor maps}.
As $\PER_B(f)=\PER_B(g)$, it follows from Fact~\ref{fact: periodic orbits of basic sets} that $\PER(F_k)=\PER(G_k)$.
Consequently, Corollary~\ref{cor: dense periodic orbits determine second iterate} implies that $F_k^2=G_k^2$---note that $\overline{\Per(F_k)}=\I$ since $\overline{\Per(f)\cap B}=B$.
As $\phi_k$ is $1$-to-$1$ onto a (dense) subset $A_k\ssq \I$ with $\overline{\phi_k^{-1}[(F_k^2)^{-1}A_k]}=f^kJ\cap B$---and as $\phi_k$ is, trivially, also a semiconjugacy between $\left.f^{2p}\right|_{f^kJ\cap B}$ and $F_k^2$ and likewise, between $\left.g^{2p}\right|_{f^kJ\cap B}$ and $G_k^2$---we obtain 
$\left.f^{2p}\right|_{\phi_k^{-1}[(F_k^2)^{-1}A_k]}=\left.g^{2p}\right|_{\phi_k^{-1}[(F_k^2)^{-1}A_k]}$ and the desired equality.

It remains to consider $p=2p'$ (the case $2p=p'$ is similar).
In this case, if $a_1,\dots,a_{p-1}$ are $\left.f\right|_B$-separators, then 
$a_2,a_4,\dots,a_{p-2}$ are $\left.g\right|_B$-separators; this follows similar to how we argued that $\left.f\right|_B$-separators are $\left.g\right|_B$-separators in the previous case.
Consequently, any $g^{p'}$-periodic point in $g^{\tau(k)}J' \cap B$ ($k \in \{0, \dots, p'-1\}$) is of even period under $g^{p'}$; in fact, each such orbit intersects $(a_{2k}, a_{2k+1}) \cap B$ as often as it intersects $(a_{2k+1}, a_{2(k+1)}) \cap B$.

Now, let $\Phi_{\tau(k)}\:g^{\tau(k)}J'\to [0,1]$ ($k\in\{0,\dots,p'-1\}$) be a continuous monotone semiconjugacy between $\left.g^{p'}\right|_{g^{\tau(k)}J'}$ and some transitive interval model $G_{\tau(k)}$ as described before Fact~\ref{fact: periodic orbits of basic sets}.
By Fact~\ref{fact: periodic orbits of basic sets} and since all $\left.g^{p'}\right|_{g^{\tau(k)}J'\cap B}$-periodic orbits are of even period, $G_{\tau(k)}$ admits a fixed point if and only if $\Phi_{\tau(k)}$ identifies the elements of a $2$-periodic orbit of $\left.g^{p'}\right|_{g^{\tau(k)}J'\cap B}$. 
Crucially, the endpoints of the gap containing $a_{2k+1}$ constitute the unique 
$2$-periodic orbit available for this identification. 
Indeed, since every $2$-periodic orbit intersects both 
$(a_{2k},a_{2k+1})\cap B$ and $(a_{2k+1},a_{2(k+1)})\cap B$, any other
such orbit must have points on either side of this gap, ruling out 
the possibility that it consists merely of the endpoints of a single gap.
Let $c^{\tau(k)}$ denote the unique fixed point of $G_{\tau(k)}$ and observe that $c^{\tau(k)}\in (0,1)$.

From \cite[Step~5, p.~40]{Blokh1995}, we know that
$g^{p'}$ maps endpoints of a gap of $B\cap g^kJ'$ to (not necessarily onto) the endpoints of some gap or to the left-most or right-most point in $B\cap g^kJ'$.
As a consequence, if a $g^{p'}$-periodic orbit in $B\cap g^kJ'$ contains both endpoints of some gap in $B\cap g^kJ'$, then any of its elements is an endpoint of a gap and moreover, the other endpoint of that gap also is an element of that orbit.
In fact, along such orbit, we have that $g^{p'}$ maps gaps onto gaps.
As a consequence of \eqref{eq: only endpoints of gaps are identified}, we obtain that
every $G_k$-periodic orbit in $[0,1]$ has as many elements to the left of $c^k$ as it has to the right of $c^k$; that is, in particular, $\Per_r(G_k)$ is non-empty only for $r=1$ or $r$ even.
Consequently, while $G_k$ is transitive, it is not mixing, since its set of periodic points of prime period is not dense; see the proof of Corollary~\ref{cor: mixing maps are fully determined by periodic orbits}.

Now, Theorem~\ref{thm: transitive vs mixing maps} gives that $[0,c^k]$ and $[c^k,1]$ are invariant under $G_k^2$ and that $G_k^2$ is mixing on both these intervals.
For $k=0,\dots,p'-1$, we define 
\[
\phi_{\tau(k);0}\: (a_{2k},a_{2k+1})\cap B \to [0,c^{\tau(k)}], \qquad x\mapsto \Phi_{\tau(k)}(x)\] 
and 
\[\phi_{\tau(k);1}\: (a_{2k+1},a_{2(k+1)})\cap B \to [c^{\tau(k)},1], \qquad x\mapsto \Phi_{\tau(k)}(x).\]
With Fact~\ref{fact: all collapsing maps are factor maps}, we see that $\phi_{\tau(k),i}$ serves as a semiconjugacy between $\left.f^p\right|_{f^{\sigma(2k+i)}J\cap B}$ and some transitive endomorphism $F_{\sigma(2k+i)}$ on $[0,c^{\tau(k)}]$ (for $i=0$) and $[c^{\tau(k)},1]$ (for $i=1$), respectively.
From Fact~\ref{fact: periodic orbits of basic sets}, the definition of $\phi_{\tau(k),i}$, and since $\PER_B(f)=\PER_B(g)$, we obtain that $F_{\sigma(2k+i)}$ has the same periodic orbits as $G_{\tau(k)}^2$ restricted to $[0,c^{\tau(k)}]$ or $[c^{\tau(k)},1]$, respectively.
By Corollary~\ref{cor: mixing maps are fully determined by periodic orbits}, this implies
$F_{\sigma(2k+i)}=G_{\tau(k)}^2$ on the respective interval.
Similarly as in the case $p=p'$, this equality lifts to $\left.f^p\right|_{f^{2k+i}J\cap B}$ and $\left.g^p\right|_{f^{2k+i}J\cap B}$ for $k=0,\ldots,p'-1$ and $i=0,1$.
This finishes the proof.
\end{proof}

We close this section by remarking that, since the topological entropy 
$h_{\text{top}}(f)$ of an endomorphism $f$ on a compact interval $\I$ is given by the supremum 
of the topological entropies over its basic sets, and recalling that 
$h_{\text{top}}(f^q)=q \cdot h_{\text{top}}(f)$, we obtain 
\begin{cor}
Suppose $f$ and $g$ are endomorphisms on a compact interval $\I$ with $\PER(f)=\PER(g)$.
Then $h_{\textrm{top}}(f)=h_{\textrm{top}}(g)$.
\end{cor}
\section{Examples}\label{sec: examples}
Distinct interval maps can, in general, possess the same periodic orbits even when their periodic points are dense. 
We present two such classes of examples on the unit interval $[0,1]$: in the first, the maps are conjugate, whereas in the second, they are not.
\begin{exmp}\label{exmp: transitive and conjugate maps with same spectrum}
 Let $f\:[0,1/2]\to[0,1/2]$ be mixing with $f(1/2)=1/2$.
 On $[0,1]$, we define
 \begin{align*}
  F_1(x)=
  \begin{cases}
   1-f(x) & \text{if } x\in [0,1/2],\\
   1-x & \text{if } x\in (1/2,1]
  \end{cases}
  \qquad \text{and} \qquad
  F_2(x)=
  \begin{cases}
   1-x & \text{if } x\in [0,1/2],\\
   f(1-x) & \text{if } x\in (1/2,1].
  \end{cases}
 \end{align*}

Note that $F_1$ and $F_2$
are conjugate via $h\:[0,1]\ni x\mapsto 1-x\in [0,1]$, that is, $h\circ F_1=F_2\circ h$.
A straightforward computation further yields
\begin{align}\label{eq: second iterates coincide}
F_1^{2n}(x)=F_2^{2n}(x)=
\begin{cases}
   f^n(x) & \text{if } x\in [0,1/2],\\
   1-f^n(1-x) & \text{if } x\in (1/2,1],
\end{cases}
\end{align}
and hence,
 \begin{align*}
  F_1^{2n+1}(x)=
  \begin{cases}
   1-f^{n+1}(x) & \text{if } x\in [0,1/2],\\
   f^n(1-x) & \text{if } x\in (1/2,1]
  \end{cases}
  \quad\ \text{and} \quad\
  F_2^{2n+1}(x)=
  \begin{cases}
   1-f^n(x) & \text{if } x\in [0,1/2],\\
   f^{n+1}(1-x) & \text{if } x\in (1/2,1],
  \end{cases}
 \end{align*}
 for all $n\geq 0$.
  It follows that $F_1$ and $F_2$ are transitive non-mixing maps; see also Theorem~\ref{thm: transitive vs mixing maps}.

By direct inspection (or Theorem~\ref{thm: transitive vs mixing maps}), we see that all periodic points of $F_1$ and $F_2$, with the exception of the fixed point $1/2$, have even period.  
From \eqref{eq: second iterates coincide}, we hence obtain that $F_1$ and $F_2$ have the same $p$-periodic points
for each $p=2n$.
In fact, $F_1$ and $F_2$ have the same periodic orbits:
a $2n$-periodic point $x\in[0,1/2)$ has an $F_1$-orbit
$\mc O_{F_1}(x)=\{x,1-f(x),f(x),1-f^2(x),f^2(x),\ldots,1-f^n(x)\}$
and an $F_2$-orbit
$\mc O_{F_2}(x)=\{x,1-x,f(x),1-f(x),f^2(x),\ldots,1-f^{n-1}(x)\}$.
As $x=F_1^{2n}x=f^n(x)$, we see that indeed, $\mc O_{F_1}(x)=\mc O_{F_2}(x)$.
\end{exmp}

In light of the previous example, it is natural to ask whether transitive interval maps with the same periodic orbits are necessarily conjugate. 
A subtler question is whether the situation of Example~\ref{exmp: transitive and conjugate maps with same spectrum} exhausts all possible pairs of transitive endomorphisms on $[0,1]$ with the same periodic orbits; that is, whether for any two such endomorphisms $f_1$ and $f_2$ there exists a self-homeomorphism $h$ of $[0,1]$ such that $h\circ f_1\circ h^{-1}$ and $h\circ f_2\circ h^{-1}$ take the same form as $F_1$ and $F_2$ in the above example.
We leave these questions open and conjecture that the answer to both is affirmative.

Building on Example~\ref{exmp: transitive and conjugate maps with same spectrum}, it is not difficult to obtain examples of pairs of non-conjugate interval maps with the same dense collection of periodic orbits.

\begin{exmp}\label{exmp: non-conjugate maps with same spectrum}
Let $g\:[0,1/4]\to [0,1/4]$ be mixing with $g(1/4)=1/4$;
let $f\:[1/4,1/2]\to[1/4,1/2]$ be mixing with $f(1/4)=1/4$ and $f(1/2)=1/2$.
For simplicity, we assume that for each period $p$, both $f$ and $g$ possess only finitely many $p$-periodic points.
 
On $[0,1]$, we define
 \begin{align*}
  G_1(x)=
  \begin{cases}
   1-g(x) & \text{if } x\in [0,1/4),\\
   1-f(x) & \text{if } x\in [1/4,1/2],\\
   1-x & \text{if } x\in (1/2,3/4],\\
   1-x & \text{if } x\in (3/4,1],
  \end{cases}
  \qquad \text{and} \qquad
  G_2(x)=
  \begin{cases}
   1-g(x) & \text{if } x\in [0,1/4),\\
   1-x & \text{if } x\in [1/4,1/2],\\
   f(1-x) & \text{if } x\in (1/2,3/4],\\
   1-x & \text{if } x\in (3/4,1].
  \end{cases}
 \end{align*}
Observe that the restrictions of $G_1$ and $G_2$ to $[1/4,3/4]$ are similar to $F_1$ and $F_2$ from the previous example, respectively (see also Figures \ref{fig: transitive and conjugate maps with same spectrum with reverse tent map} and \ref{fig: non-transitive and non-conjugate maps with same spectrum}).
As $\left.G_1\right|_{[0,1/4]\cup[3/4,1]}=
\left.G_2\right|_{[0,1/4]\cup[3/4,1]}$, we hence have that $G_1$ and $G_2$ have the same periodic orbits.

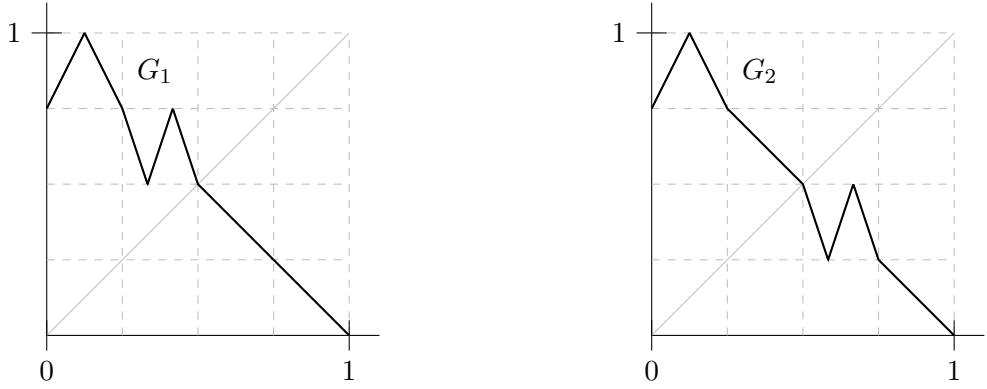
\begin{figure}[h]
    \centering

    \begin{tikzpicture}[scale=4.0, declare function={
        g(\x)=(1/4) * abs(8 * \x - 1);
        f(\x)=(1/4) * (abs(abs(12 * \x - 4)-1)+1); 
        S1(\x)= ifthenelse(\x < 1/4, 1 - g(\x), ifthenelse(\x < 1/2, 1 - f(\x), 1 - \x));
        S2(\x)= ifthenelse(\x < 1/4, 1 - g(\x), ifthenelse(\x < 1/2, 1 - \x, ifthenelse(\x <= 3/4, f(1 - \x), 1 - \x)));
    }]
        \begin{scope}[local bounding box=left]
            \draw[-] (0,0) -- (1.1,0);
            \draw[-] (0,0) -- (0,1.1);
            \draw[lightgray] (0,0) -- (1,1);
            \draw[dashed, lightgray] (0,0.5) -- (1,0.5);
            \draw[dashed, lightgray] (0,1) -- (1,1);
            \draw[dashed, lightgray] (1,0) -- (1,1);
            \draw[dashed, lightgray] (0.5,0) -- (0.5,1);
            \draw[dashed, lightgray] (0.25,0) -- (0.25,1);
            \draw[dashed, lightgray] (0.75,0) -- (0.75,1);
            \draw[dashed, lightgray] (0,0.25) -- (1,0.25);
            \draw[dashed, lightgray] (0,0.75) -- (1,0.75);
            \draw[thick] plot[domain=0:1, samples=500] (\x, {S1(\x)});
            \foreach \X in {0,1}
            {\draw (\X,0.05) -- (\X,-0.05) node[below]{$\X$};}
            \draw (0.05,1) -- (-0.05,1) node[left]{$1$};
            \node[above left] at (0.45,0.8) {$G_1$};
        \end{scope}
    
        \begin{scope}[xshift=2.0cm, local bounding box=right]
            \draw[-] (0,0) -- (1.1,0);
            \draw[-] (0,0) -- (0,1.1);
            \draw[lightgray] (0,0) -- (1,1);
            \draw[dashed, lightgray] (0,0.5) -- (1,0.5);
            \draw[dashed, lightgray] (0,1) -- (1,1);
            \draw[dashed, lightgray] (1,0) -- (1,1);
            \draw[dashed, lightgray] (0.5,0) -- (0.5,1);
            \draw[dashed, lightgray] (0.25,0) -- (0.25,1);
            \draw[dashed, lightgray] (0.75,0) -- (0.75,1);
            \draw[dashed, lightgray] (0,0.25) -- (1,0.25);
            \draw[dashed, lightgray] (0,0.75) -- (1,0.75);
            \draw[thick] plot[domain=0:1, samples=500] (\x, {S2(\x)});
            \foreach \X in {0,1}
            {\draw (\X,0.05) -- (\X,-0.05) node[below]{$\X$};}
            \draw (0.05,1) -- (-0.05,1) node[left]{$1$};
            \node[above left] at (0.45,0.8) {$G_2$};
        \end{scope}
    \end{tikzpicture}
    \caption{
        Two endomorphisms $G_1$ and $G_2$ on the unit interval having dense periodic points and the same period spectrum without being conjugate (here, $g(x) = 1/4\cdot |8x - 1|$ on $[0, 1/4]$ and $f(x)=1/4 \cdot \big(\big||12x - 4|-1\big|+1\big)$ on $[1/4,1/2]$).
    }
    \label{fig: non-transitive and non-conjugate maps with same spectrum}
\end{figure}

However, $G_1$ and $G_2$ are not conjugate.
To see this, assume for a contradiction that $h\:[0,1]\to[0,1]$ is a conjugacy between the two systems.
Note that $h$ has to map $p$-periodic points to $p$-periodic points.
If $h$ is order-preserving,
from the assumption of only finitely many periodic points of a given period and as $\Per(G_1)=\Per(G_2)$, we obtain that $h$
is necessarily the identity on all periodic points and hence, by denseness, the identity on $[0,1]$; contradicting $G_1\neq G_2$.

If $h$ is order-reversing, however, note that $h$ maps $J=[1/4,3/4]$ onto itself and hence, $h[0,1/4]=[3/4,1]$---this follows since $J$ is the unique proper invariant subinterval of $G_1$ as well as $G_2$.
Consequently, considering $h\circ G_1=G_2\circ h$ on $[0,1/4]$, we obtain
$h\circ (1-g)=1-h$ which is impossible as the right-hand side is a bijection while the left-hand side is not.
\end{exmp}

\section{Proof of Lemma~\ref{lem: local type of fixed points}}
\label{sec: proof of technical lemma}
We use the notation and terminology from Section~\ref{sec: periodic points on the interval determine left-right behaviour}.
\begin{proof}[Proof of Lemma~\ref{lem: local type of fixed points}]
\mbox{}
\begin{enumerate}[label=(\roman*),leftmargin=*]
\item Let $\llbracket a,b\rrbracket$ be the connected component of $\Fix_p(f)$ containing $x$. 
We prove the statement by showing that $\llbracket a,b\rrbracket$ is also a connected component of $\Fix_{ip}(f)$, and that the types of $a$ and $b$ (if applicable, i.e., if $a>\0$ and $b<\1$, respectively) are independent of $i$.
Note that since $\llbracket a, b \rrbracket \subseteq \Fix_{ip}(f)$ holds trivially, the former reduces to showing that $a \in \partial \Fix_{ip}(f)$ if $a \neq \0$, and likewise, $b \in \partial \Fix_{ip}(f)$ if $b \neq \1$.
This amounts to showing that $a$ or $b$, respectively, are accumulated by points at which $f^{ip}$ lies strictly above or below the diagonal.

 We restrict attention to the types of $a$ (hence, we specifically consider $a>\0$); preservation of the types of $b$ follows similarly.
 If $a$ is of type $(-,-)$ or $(+,+)$, the statement follows from \ref{item: types of fixed points 2} (proven below).
Note further that $a$ cannot be of type $(+,\pm)$ or $(-,\pm)$ under $f^p$: if it were then, as the left endpoint of a component of $\Fix_p(f)$, it would admit a $(+,+)$ or $(-,-)$ neighbourhood, respectively, contradicting the definition of $(+,\pm)$ or $(-,\pm)$.
 Fixing $i\in \N$, we go through the remaining cases.

Suppose $a$ is of type $(+,-)$ under $f
^p$ and let $U$ be a corresponding neighbourhood of $a$; that is, in particular, $f^p(y)\geq y$ for all $y\in U\cap \llbracket\0,a]$ (i.e.\ for all $y$ to the left of $a$) and  $f^p(y')>y'$ for some $y'\in U\cap\llbracket\0,a]$; it is worth remarking that $y'$ can be chosen arbitrarily close to $a$ for otherwise, $a\in \mathring \Fix_p(f)$.
Note that by replacing $U$ by $U\cap [y_-,\1 \rrbracket$ for some appropriate $y_-\in U\cap \llbracket \0,a)$, we may assume without loss of generality that $f^p\big(U\cap \llbracket \0,a]\big)\ssq U\cap \llbracket \0,b\rrbracket$ (in other words, no point in $U$ lying to the left of $a$ is mapped to the right of $\llbracket a, b \rrbracket$).
But then, as $\left.f^p\right|_{\llbracket a,b\rrbracket}=\Id$, we have $f^{ip}(y)\geq y$ for all $y\in U\cap \llbracket \0,a]$, and hence $a$ is of type $(+,\ast)$ under $f^{ip}$, showing that the first coordinate of its type is preserved. 
Note that $a\in\partial\Fix_{ip}(f)$ 
indeed holds, since $y'$ can be chosen arbitrarily close to $a$ and satisfies $f^{ip}(y')>y'$.

A similar argument to the above shows that if
$a$ is of type $(+,-)$ or $(\pm,-)$ under $f^p$, then $a$ is of type $(\ast,-)$ under $f^{ip}$, that is, the second coordinate of its type is preserved.

If $a$ is of type $(-,+)$ 
under $f^p$, then it follows from \ref{item: types of fixed points 3} (proven below) that $a$ is of type $(-,+)$ under $f^{ip}$.
Note that \ref{item: types of fixed points 3} also gives that if $a$ is of type $(\pm,+)$ under $f^p$, then it is of type $(\ast,+)$ under $f^{ip}$.
Altogether, the above demonstrates that whenever a coordinate of the type of $a$ under $f^{p}$ is $+$ or $-${}, it matches the corresponding coordinate of the type of $a$ under $f^{ip}$.

It therefore remains to consider the cases where $a$ is of type $(\pm, \ast)$ or $(\ast, \pm)$; while irrelevant for what follows, note that in the latter case, $a$ must necessarily be of type $(\pm, \pm)$.
We focus on the case where $a$ is of type $(\pm,\ast)$ under $f^p$; the other case follows analogously.
Let $U$ be a corresponding neighbourhood of $a$; in particular, there exist $y_-^-,y_-^+\in U\cap\llbracket \0,a]$ with $f^p(y_-^-)<y_-^-$ and $f^p(y_-^+)>y_-^+$.
Let $B$ be the connected component of $U$ that contains $y_-^-$ and where $f^p(y)<y$ for all $y\in B$.
The right endpoint $b'$ of $B$ is fixed and hence, we can find $y^-\in B$ sufficiently close to $b'$ to ensure that $f^{jp}(y^-)\in B$ for $j=1,\ldots,i$.
Similarly, let $C$ denote the connected component of $U$ that contains $y_-^+$ and where $f^p(y)>y$ for all $y\in C$.
Assume without loss of generality that $y_-^+>y_-$ (if that was impossible, then $a$ would be of type $(-,\ast)$ which is ruled out by the definition of $(\pm,\ast)$).
The left endpoint $c$ of $C$ is fixed
and hence, we can find $y^+\in C$ sufficiently close to $c$ to ensure that $f^{jp}(y^+)\in C$ for $j=1,\ldots,i$.
Consequently, we obtain that $a$ is of type $(\pm,\ast)$ under $f^{ip}$; let us emphasize that $a\in \partial \Fix_{ip}(f)$ since $y^-$ and $y^+$ can be chosen arbitrarily close to $a$.

Combining all of the above, we conclude that if $a$ is of type $(\alpha,\beta)$ under $f^p$, then it is also of type $(\alpha,\beta)$ under $f^{ip}$, thereby completing the proof of \ref{item: types of fixed points 1}.

Note that the argument in the last part of the above proof did not use the fact that $a$ is an element of $\partial\,{\mathring\Fix_p(f)}$.
That is, arguing as above one can actually show 
the following implication for all $x\in \partial\Fix_p(f)$ and $i\in\N$,
\begin{align}\label{eq: pm stays pm}
x \text{ is of type } (\pm,\ast) \text{ [resp. } (\ast,\pm)\text{] under } f^p\Longrightarrow
x \text{ is of type } (\pm,\ast) \text{ [resp.} (\ast,\pm)\text{] under } f^{ip}.
\end{align}
We will use \eqref{eq: pm stays pm}  when proving \ref{item: types of fixed points 4} and \ref{item: types of fixed points 6}(c).
\item Suppose $x$ is of type $(+,+)$ under $f^p$ and pick some neighbourhood $U$ of $x$ with $f^p(y)\geq y$ for all $y\in U$.
        As $f^p(x)=x$ and as $f$ is continuous, given $i\geq 2$, there is $V\ssq U$ with $f^{jp}(V)\ssq U$ for $j=1,\dots,i$.
        In particular, as $x$ is of type $(+,+)$, for each $y\in V$ we have $f^{jp}(y)\geq y$.
        Further, since $x\in\partial \Fix_p(f)$, there is $y\in V$ arbitrarliy close to $x$ with $y\notin \Fix_p(f)$, that is, $f^p(y)>y$ and hence, $f^{ip}(y)>y$.
        It follows that $x\in \partial \Fix_{ip}(f)$ so that altogether, $x$ is of type $(+,+)$.
        The case with $x$ of type $(-,-)$ works similarly.
        \item We only consider the case $(\ast,+)$, the other case is similar.
        Note that if $x$ is of type $(+,+)$, the statement follows from \ref{item: types of fixed points 2}.
        We may hence assume without loss of generality that $x$ is of type $(-,+)$ or $(\pm,+)$ and that $x$ admits no $(+,+)$-neighbourhood.
        Then there is a neighbourhood $U$ of $x$ with $f^p(y)\geq y$ for all $y\geq x$ in $U$ and there are $y_-<x<y_+$ in $U$ with $f^p(y_-)<y_-$ and $f^p(y_+)>y_+$.
        In fact, $y_-$ can be chosen arbitrarily close to $x$ as we assume $x$ to not be of type $(+,+)$.
        
        Consider $i\in \N$.
        We recursively define a finite sequence of closed intervals $[x,y_+]=A_1\supseteq A_2\supseteq \dots\supseteq A_i$ by setting $A_{j+1}$ (for $j=1,\ldots,i-1$) to be the connected component of $A_j\cap f^{-p}A_j$ that contains $x$.
        It is easy to see that $f^{ip}(y)\geq y$ for all $y\in A_i$ and that there is $y_+'\in A_i$ with $f^{ip}(y_+')>y_+'$.

        Now, let $y_-$ be as above and let $B$ denote the connected component of $U$ that contains $y_-$ and where $f^p(y)< y$ for all $y\in B$.  
        The right endpoint $b$ of $B$ lies in $\Fix_p(f)$.
        Hence, by continuity of $f$, there is $y_-'$ sufficiently close to $b$ such that $f^{jp}(y')\in B$ for $j=1,\dots,i$.
        In particular, $f^{ip}(y_-')<y_-'$.
        Moreover, since $y_-$ can be chosen arbitrarily close to $x$, it follows that $y_-'$ can also be chosen arbitrarily close to $x$, which shows that $x \in \partial \Fix_{ip}(f)$.
        Further, $[y_-',y_+']$ is a neighbourhood of $x$ where $f^{ip}(y)\geq y$ for all $y\geq x$ (and with $f^{ip}(y_+')> y_+'$ and $f^{ip}(y_-')< y_-'$).
        This shows that $x$ is of type $(-, +)$ or $(\pm, +)$; in particular, $x$ is of type $(\ast, +)$.
        \item 
            Consider $i\in\N$. Due to \eqref{eq: pm stays pm},  if $x$ is of type $(\pm,\ast)$ under $f^p$ then so it is under $f^{ip}$.
            It therefore suffices to show that $x$ is of type $(\ast,-)$ or $(\ast,\pm)$ under $f^{ip}$.

Let $U$ be a $(\pm,-)$-neighbourhood of $x$ under $f^p$; in 
particular, there exists $y_+\in U\cap[x,\1 \rrbracket$ with $f^p(y_+)<y_+$, 
and we may assume $y_+$ can be chosen arbitrarily close to $x$ (the case where $x$ is the left endpoint of a connected component of $\mathring{\Fix}_p(f)$ being already addressed in 
\ref{item: types of fixed points 1}). 
Note also that $x$ is 
accumulated from the left by fixed points of $f^p$, as it is of 
type $(\pm,-)$.

If there exists $y\in U\cap(x,y_+ ]$ with $f^p(y)\leq x$, then there 
necessarily exists $y'\in(x,y]$ with $f^p(y')\leq x <y'$ and 
$f^p(y')\in\Fix_p(f)$, giving $f^{ip}(y')<y'$. 
If instead $f^p(y)>x$ for all $y\in U\cap(x,y_+]$, then 
$f^p\big(U\cap [x,y_+]\big)\ssq U\cap [x,y_+]$, and one readily obtains $f^{ip}(y)\leq y$ for all 
$y\in U\cap(x,y_+]$ and $f^{ip}(y_+)<y_+$.
In either case, we find that $x$ is accumulated from the right by points at which $f^{ip}$ lies strictly below the diagonal; the statement follows.

\item 
        This is similar to \ref{item: types of fixed points 4}.
        \item In all of the following, $U$ is a $(+,-)$-neighbourhood of $x$ (under $f^p$).
        Note that we may choose $U$ to be arbitrarily small; if $U$ contained a neighbourhood of $x$ that was not a $(+,-)$-neighbourhood, then $x \in \partial \mathring{\Fix}_p(f)$, which is already addressed in \ref{item: types of fixed points 1}.

We begin by considering alternative (d).
To this end, suppose that $x \in \mathring{\Fix}_{2p}(f)$. 
By restricting to a sufficiently small neighbourhood, we have 
$f^{2p}(y)=y$ for all $y\in U$; and hence $f^{2ip}(y)=y$ for all 
$y\in U$ and $i\in\N$.
Furthermore, for all $y \in U$ and every $i \in \N$,
\(
f^{(2i+1)p}(y) = f^p(f^{2ip}(y)) = f^p(y).
\)
Accordingly, $x$ is of type $(+,-)$ under $f^{(2i+1)p}$.
This corresponds precisely to alternative (d).

Towards alternative (b), suppose there exists a $(\ast,+)$-neighbourhood $V$ of $x$ under $f^{2p}$.
It follows with \ref{item: types of fixed points 3} that $x$ is of type $(\ast,+)$ under $f^{ip}$ for all even $i$. 
Further, by passing to a smaller neighbourhood if necessary, we may assume that $f^p(y) \leq y$ and $f^{2p}(y)\geq y$ for all $y \in V \cap [x,\1 \rrbracket$, and that there is $y_+\in V\cap [x,\1 \rrbracket$ with $f^p(y_+)<y_+$.
Observe that this is only possible if $f^p(y_+) < x$; otherwise, we would have $f^p(y_+) \in V\cap [x, \1 \rrbracket$ and hence $f^{2p}(y_+) = f^p(f^p(y_+)) \leq f^p(y_+) < y_+$.
In fact, it follows that $f^p(y) < x$ for all $y \in V \cap (x, \1 \rrbracket$.
Now, given $i \in \N$, we may choose a neighbourhood $U \subseteq V$ of $x$ such that $f^{jp}U \subseteq V$ for $j = 1, \dots, i$.
It then follows directly that
\(f^{jp}y < x < y\) for every $y\in U\cap (x,\1 \rrbracket$ and odd \(j \leq i\).
Summing up, the above shows that if $x$ is of type $(\ast,+)$ under $f^{2p}$, then we are necessarily in alternative (b).

        As for alternative (a), assume now that $f^{2p}(y)\leq y$ for every $y\in U\cap(x,\1 \rrbracket$. 
        Note that if there is a neighbourhood $U$ of $x$ with $f^p\big(U\cap[x,\1 \rrbracket\big)\ssq [x,\1 \rrbracket$ (in other words, $f^p(y)\geq x$ for each $y\in U\cap[x,\1 \rrbracket$), it is easy to see that $x$ is of type $(\ast,-)$ under $f^{ip}$ for all $i\in \N$.
        To address the remaining possibilities that may correspond to alternative (a), we need
        \begin{claim}\label{claim: proof technical lemma}
        Given $i\in\N$, there exists a neighbourhood $V\ssq U$ of $x$ with $f^{ip}(y)\leq y$ for every 
        $y\in V\cap(x,\1 \rrbracket$.
        \end{claim}
         \begin{proof}[Proof of Claim~\ref{claim: proof technical lemma}]
        Assume for a contradiction that there exists $i\in \N$ such that there are points $y>x$ arbitrarily close to $x$ with $f^{ip}(y)>y$.
        Let $V\ssq U$ be a $(+,-)$-neighbourhood of $x$ under $f^p$ such that $f^{jp}V\ssq U$ for $j=1,\dots,i$.
        Let further $i'\leq i$ be minimal such that there is $y\in V\cap (x,\1 \rrbracket$ with $f^{i'p}(y)>y$; note that we necessarily have $f^{(i'-1)p}(y)<x$.
        Since $V \subseteq U$, our assumptions---specifically, the assumption that $f^{2p}(y) \leq y$ for every $y \in U \cap (x, \1 \rrbracket$---imply that $i' > 2$.
        Observe that if we had $f^{(i'-2)p}(y)>x$ (that is, if $f^{(i'-2)p}(y)\in V\cap(x,\1 \rrbracket$), then with $y'=f^{(i'-2)p}(y)\leq y$ we would obtain $f^{2p}(y')=f^{i'p}(y)>y\geq y'$, contradicting our assumptions.
        Hence, we necessarily have $f^{(i'-2)p}(y)<x$.
\newline
        \indent
        Let $k > 1$ be minimal such that $f^{(i'-k)p}(y) > x$. 
        As $f^{(i'-(k-1))p}(y) < \ldots < f^{(i'-1)p}(y)$ (which holds since $f^{jp}V\ssq U$ for all $j=1,\ldots,i$ and $f^p(z) \geq z$ for all $z \in U \cap \llbracket \0,x)$), there necessarily exists $y'' \in [x, f^{(i'-k)p}(y)] \subseteq [x,y]$ such that $f^p(y'') = f^{(i'-1)p}(y)$. This yields $f^{2p}(y'') = f^{i'p}(y) > y \geq y''$, contradicting our assumptions.
         \end{proof}
         With the above claim and the discussion preceding it, it only remains to show that if there exist points $y' \in (x, \1 \rrbracket$ arbitrarily close to $x$ such that $f^p(y') < x$, then $x$ is accumulated from the right by points at which $f^{ip}$ lies strictly below the diagonal.
        \newline\indent
        Now, if there also exist points $y \in (x, \1 \rrbracket$ arbitrarily close to $x$ such that $f^p(y) \geq x$, then $x$ is necessarily accumulated by points $y'' \in (x, \1 \rrbracket$ satisfying $f^p(y'') = x$.
        It follows $f^{ip}(y'') = x < y''$ for all $i \in \N$, so $x$ is indeed of type $(\ast, -)$ under $f^{ip}$ in this case.
\newline
\indent
   We may therefore assume that $f^p(y) < x$ for all $y \in U \cap (x, \1 \rrbracket$.
In this case, if $x$ were not accumulated by points $y \in (x, \1 \rrbracket$ such that $f^{2p}(y) < y$, then we would have $f^{2p}(y) = y$ for all $y \geq x$ sufficiently close to $x$ (Claim~\ref{claim: proof technical lemma}), and therefore
\[
f^{2p}(f^p(y)) = f^{p}(f^{2p}(y)) = f^p(y)
\]
for all such $y$.
This implies that in fact $f^{2p}(y) = y$ on a neighbourhood of $x$, that is, $x \in \mathring{\Fix}_{2p}(f)$, which corresponds to alternative~(d) as already discussed above.
Accordingly, there must exist points $y \in (x, \1 \rrbracket$ arbitrarily close to $x$ such that $f^{2p}(y) < y$.
Given $i \in \N$, Claim~\ref{claim: proof technical lemma} implies that $f^{ip}(y) < y$ for any such $y$ sufficiently close to $x$.
Altogether, we conclude that if $x$ is of type $(\ast, -)$ under $f^{2p}$, then we are in alternative~(a).
\newline
\indent
Finally, let us consider alternative (c), that is, suppose $x$ is of type $(\ast,\pm)$ under $f^{2p}$.
From \eqref{eq: pm stays pm} we obtain that $x$ is of type $(\ast,\pm)$ under $f^{2pi}$ for every $i\in\N$.
\newline
\indent
Moreover, if $x$ is of type $(\ast,\pm)$ under $f^{2p}$, then $x$ 
is accumulated from the right by fixed points of $f^{2p}$, and at 
least one of the following holds: $x$ is accumulated from the right by fixed points of $f^p$, or $x$ is accumulated from the right by fixed points $y'$ of $f^{2p}$ satisfying $f^p(y')<x<y'$ (to see this, use that $x$ is of type $(\ast,-)$ under $f^p$).
\newline
\indent
In the latter case, $x$ is clearly accumulated from the right by 
points $y'$ with $f^{ip}(y')<y'$ for every odd $i$. 
In the former case, for any $y \in\d \Fix_p(f)\cap (x,\1 \rrbracket$ sufficiently close to $x$, one can find points $z$ arbitrarily close to $y$ such that $f^p(z)<z$.
Consequently, for any $i \in \N$ and any neighbourhood $U$ of $x$, there exists $z \in \big(U \cap (x, \1 \rrbracket\big) \setminus \Fix_p(f)$ such that $f^{jp}(z) \in U \cap (x, \1 \rrbracket$ and hence $f^{jp}(z) < z$ for all $j = 1, \dots, i$.
It follows that 
$x\in \partial\Fix_{ip}(f)$ and that $x$ is not of type $(\ast,+)$ under $f^{ip}$ for odd $i$. 
This is alternative~(c).
        \item This is similar to \ref{item: types of fixed points 6}.\qedhere
\end{enumerate}
\end{proof}

\appendix
\section{Dense periodicity on non-compact intervals}\label{sec: appendix}
In this appendix, we discuss how the proof of Theorem~\ref{thm: CovenMulvey}, given in \cite{CovenMulvey1988} for compact intervals, extends to arbitrary intervals.

We briefly recall some notation and terminology.
Let $\I = \llbracket \0, \1 \rrbracket$ be an arbitrary interval, where $\rrbracket$ denotes either $)$ or $]$, $\llbracket$ denotes either $($ or $[$, $\0 \in \R \cup \{-\infty\}$, and $\1 \in \R \cup \{\infty\}$ with $\1 > \0$.
Recall that an endomorphism $f \colon \I \to \I$ \emph{separates} $\I$ if $\I$ admits a decomposition $\I = L \cup C \cup R$ into relatively closed intervals with pairwise disjoint interiors, arranged from left to right in $\I$, such that $L$ and $R$ are non-degenerate (while $C$ may be degenerate), $f(L) \ssq R$, $f(C) = C$, and $f(R) \ssq L$.
Note that if $f$ is surjective, the inclusions are necessarily equalities, that is, $f(L) = R$ and $f(R) = L$.

Following the terminology of \cite{CovenMulvey1988}, we call an endomorphism $f$ on $\I$ \emph{pointwise non-wandering} if for every non-empty open set $U \subseteq \I$, there exists $n \in \N$ such that $f^n(U) \cap U \neq \emptyset$.
Note that if $\Per(f)$ is dense in $\I$, then $f$ is necessarily pointwise non-wandering.
The following hence immediately implies Theorem~\ref{thm: CovenMulvey}.

\begin{thm}\label{thm: CovenMulvey appendix}
Let $f\colon \I \to \I$ be continuous and pointwise non-wandering.
If $f$ does not separate $\I$, then there exists a collection (possibly finite or empty) $J_1,J_2,\ldots$ of non-degenerate, closed (as subsets of $\I$), invariant subintervals of $\I$, with pairwise disjoint interiors, such that
\begin{itemize}
    \item $\left.f\right|_{J_i}$ is topologically mixing for each $i$, and
    \item $\left.f\right|_{\I \setminus \bigcup_i J_i}$ is the identity map.
\end{itemize}

If $f$ separates $\I$, $f^2$ does not, and the conclusion of the preceding paragraph holds for $f^2$.
\end{thm}

We now follow the proof of Theorem~\ref{thm: CovenMulvey appendix} as presented for compact $\I$ in \cite{CovenMulvey1988}, quoting relevant intermediate statements and definitions but providing proofs only where there are non-trivial differences between the compact and non-compact cases.

In Lemmas~\ref{lem: coven mulvey lemma 2}--\ref{coven mulvey lemma 6} below, $f$ is a pointwise non-wandering endomorphism which does not separate $\I$.
For a subset \( J \) of \( \I \), let \( J^* \) be the smallest closed, invariant set containing \( J \); that is,
\[
J^* = \overline{\bigcup_{k \geq 0} f^k(J)}.
\]

\begin{lem}[{cf.~\cite[Lemma~2]{CovenMulvey1988}}]\label{lem: coven mulvey lemma 2}
For every non-degenerate subinterval $J$ of $\I$, $J^*$ is a closed interval.
\end{lem}

\begin{lem}[{cf.~\cite[Lemma~3]{CovenMulvey1988}}] \label{lem: endpoints are fixed}
If $J$ is a closed, invariant interval, then any endpoint of $J$ which is interior to $\I$ (that is, an element of $\I\setminus \{\0,\1\}$) is a fixed point of $f$.
\end{lem}

Let us write
\[
\{x\in \I\colon f(x)\neq x\}=\I\setminus \Fix_1(f)
\]
as a countable union of disjoint open intervals $G_n$.

\begin{lem}[{cf.~\cite[Lemma~5]{CovenMulvey1988}}]\label{lem: transitivity on Gn}
Each $\left.f\right|_{G_n^*}$ is transitive, and hence for any $n$ and $m$, $G_n^*$ and $G_m^*$ are either equal, disjoint, or meet in exactly one point.
\end{lem}

\begin{lem}[{cf.~\cite[Lemma~6]{CovenMulvey1988}}]\label{coven mulvey lemma 6}
Each $\left.f\right|_{G_n^*}$ is topologically mixing.
\end{lem}
\begin{proof}
Due to \cite[Lemma~7.1]{Banks1997}, it suffices to show that 
$\left.f\right|_{G_n^*}$ is totally transitive (that is, every iterate of $f$ is transitive).
To this end, since $\left.f\right|_{G_n^*}$ is transitive, we only need to show that also $\left.f^2\right|_{G_n^*}$ is transitive; see \cite[Theorem~3.1]{Banks1997} and \cite[Theorem~7.1]{Banks1997}.

For a contradiction, suppose $\left.f^2\right|_{G_n^*}$ is not transitive.
Due to transitivity of $\left.f\right|_{G_n^*}$ and Theorem~\ref{thm: transitive vs mixing maps}, this implies 
$G_n^*=L\cup R$, where $L$ and $R$ are non-degenerate closed intervals with exactly one point in common, such that $f(L)=R$ and $f(R)=L$.

Now, the proof continues exactly as in the compact case in \cite{CovenMulvey1988}:
if $G_n^*\neq \I$, then at least one endpoint $p$ of $G_n^*$ is interior to $\I$. 
Then by Lemma~\ref{lem: endpoints are fixed}, $p$ is a fixed point of $f$, and so $p=f(p)\in L\cap R$. Since this is not the case, we must have $G_n^*=\I$. 
But then $f$ separates $\I=L\cup C\cup R$, where $C=L\cap R$.
\end{proof}

Altogether, the above shows the first part of Theorem~\ref{thm: CovenMulvey appendix} with $J_1,J_2,\ldots$ an enumeration of the intervals $G_n^*$.
Since $f$ being pointwise non-wandering implies that $f^2$ is also pointwise non-wandering, the second part of Theorem~\ref{thm: CovenMulvey appendix} follows by showing that $f^2$ does not separate $\I$.

Before turning to the proof of this fact, we note that pointwise non-wandering endomorphisms on a half-open interval $\I$ never separate $\I$.
To show this, we will make use of the following observation:
a pointwise non-wandering endomorphism $f$ on $\I = (\0, \1]$ or $\I = [\0, \1)$ may fail to be surjective, but we always have $(\0, \1) \subseteq f(\I)$.
In particular,
if $\I = (\0, \1]$, then $f(L) \supseteq R \cap (\0, \1)$ and $f(R) = L$, whilst if $\I = [\0, \1)$, then $f(R) \supseteq L \cap (\0, \1)$ and $f(L) = R$.

\begin{lem}
Let $\I$ be a half-open interval and let $f\colon \I\to \I$ be continuous and pointwise non-wandering.
Then $f$ does not separate $\I$.
\end{lem}
\begin{proof}
Consider $\I=[\0,\1)$ with $\0\in \R$ and $\1\in \R\cup\{\infty\}$; the case where $\I$ is left-open instead is entirely analogous.

Suppose for a contradiction that $\I=L\cup C\cup R$ witnesses $f$ separating $\I$.
As $\I=[\0,\1)$, the leftmost piece $L$ contains the endpoint $\0$, so $L=[\0,l]$ for some $l$, while the rightmost piece $R$ is of the form $R=[r,\1)$ for some $r\geq l$.
In particular, $L$ is compact while $R$ is not.
By continuity, also $f(L)$ is compact, so $f(L)\neq R$.
As a consequence of the observation prior to the statement, this shows that $f$ does not separate $\I$.
\end{proof}

We conclude the proof of Theorem~\ref{thm: CovenMulvey appendix} with Lemma~\ref{lem: second iterate is not separating} below (the compact case is already dealt with in \cite{CovenMulvey1988}).
For its proof, note that if $\I$ is open, then pointwise non-wandering endomorphisms are necessarily surjective.

\begin{lem}[{cf.~\cite[Lemma~7]{CovenMulvey1988}}]\label{lem: second iterate is not separating}
Let $\I$ be an open interval and let $f\colon \I\to \I$ be continuous and pointwise non-wandering.
Then $f^2$ does not separate $\I$.
\end{lem}
\begin{proof}
Consider $\I=(\0,\1)$, and assume for a contradiction that $f^2$ separates $\I=L\cup C\cup R$, with $f^2(L)= R$, $f^2(C)=C$, and $f^2(R)= L$ (where we have equalities since $f$ [and thus $f^2$] is onto).
Clearly, both $L=(\0,l]$ and $R=[r,\1)$ (with $r\geq l$) are half-open; $C$ is compact.

In the following, we say an interval is \emph{left-unbounded} if it is of the form $(\0, l'\rrbracket$ for some $l' \in \I$ and \emph{right-unbounded} if it is of the form $\llbracket r', \1)$ for some $r' \in \I$.

We first note that neither of the intervals $f(L)$ nor $f(R)$ can equal $\I$: if, say, $f(R) = \I$, then $f^2(R) = f(\I) = \I$, contradicting $f^2(R) = L \neq \I$; the argument for $f(L)$ is similar.
Next, note that each of $\I\setminus f(L)$ and $\I\setminus f(R)$ has only one connected component.
Suppose instead that $\I \setminus f(L)$ has two components, one left-unbounded and one right-unbounded; then $f(L)$ is contained in a compact subinterval of $\I$, and hence so is $f(L\cup C) = f(L)\cup f(C)$, as $f(C)$ is compact.
To ensure surjectivity of $f$, 
$f(R)$ would then have to contain both a left-unbounded and a right-unbounded interval; being an interval, $f(R)$ would thus equal $\I$, contradicting the above.
The same argument applies to $\I\setminus f(R)$.

Combining these observations, exactly one of the following two cases occurs: either $f(L)$ is left-unbounded and $f(R)$ is right-unbounded, or vice versa.

Suppose first that $f(L)=(\0,l'\rrbracket$ is left-unbounded (and $f(R)$ is right-unbounded).
In this case, $f$ maps any left-unbounded interval to a left-unbounded interval.
Indeed, consider $(\0, a\rrbracket$ and assume without loss of generality that $a < l$ (the other case is similar). Then,
\[
(\0, l'\rrbracket = f(L) = f\big((\0, a)\big) \cup f([a, l]),
\]
so $f\big((\0, a\rrbracket\big)$ is left-unbounded, since $f([a, l])$ is compact.
This observation, however, implies that $f^2(L)$ is also left-unbounded, which contradicts $f^2(L) = R$.

It remains to consider the case where $f(L)$ is right-unbounded and $f(R)$ is left-unbounded.
By similar reasoning to the above, any right-unbounded interval is mapped under $f$ to a left-unbounded interval.
In particular, $f^2(L) = f(f(L))$ would be left-unbounded, which contradicts $f^2(L) = R$.
This completes the proof.
\end{proof}

\bibliography{literature}{}
\bibliographystyle{abbrv}
\vspace{-0.38cm}
\end{document}